\documentclass[a4paper,12pt]{amsart}
\usepackage{amsthm,amsmath, amssymb, amscd,mathtools, braket}
\usepackage[top=30truemm, bottom=30truemm, left=25truemm, right=25truemm]{geometry}
\usepackage{color}
\usepackage{stmaryrd}
\usepackage{bbm}
\usepackage[dvipdfmx,setpagesize=false]{hyperref}
\usepackage{mathrsfs}
\usepackage{multicol}

\mathtoolsset{showonlyrefs=true}

\newcommand{\Tr}{\mathrm{Tr}} 

\numberwithin{equation}{section}
\newcounter{count}

\newcommand{\num}{\stepcounter{count}\the\value{count}}
\renewcommand{\mod}{\ \mathrm{mod}\ }

\newtheorem{theorem}{Theorem}[section]
\newtheorem{lemma}[theorem]{Lemma}
\newtheorem{corollary}[theorem]{Corollary}

\newtheorem{proposition}[theorem]{Proposition}

\newtheorem{question}[theorem]{Question}

\theoremstyle{remark}
\newtheorem{remark}[theorem]{Remark}

\theoremstyle{definition}
\newtheorem{notation}[theorem]{Notation}

\begin{document}

\title[Transcendency of Mills-type function]{Transcendency of variants of Mills' constant}

\author[K. Saito]{Kota Saito}
\address{Kota Saito\\ Department of Mathematics\\ College of Science $\&$ Technology \\ Nihon University\\Kanda\\ Chiyoda-ku\\ Tokyo\\
101-8308\\ Japan} 
\email{saito.kota@nihon-u.ac.jp}

\thanks{The previous address of the author is  ``Faculty of Pure and Applied Sciences, University of Tsukuba, 1-1-1 Tennodai, Tsukuba, Ibaraki, 305-8577, Japan'' }


\begin{abstract} Let  $\lfloor x\rfloor$ denote the integer part of $x$. For every sequence $(C_k)_{k\ge 1}$ of positive integers, we define $\xi(C_k)$ as the smallest real number $\xi>1$ such that $\lfloor \xi^{C_k} \rfloor$ is a prime number for every positive integer $k$. The number $\xi(3^k)$ is called Mills' constant. Recently, the author showed that $\xi(3^k)$ is irrational; however, the transcendency remains open. In this paper, we show that Mills' constant is transcendental under the Density Hypothesis of the Riemann zeta function. Furthermore, we obtain four classes of sequences $(C_k)_{k\ge 1}$ for which we can verify the arithmetic properties of $\xi(C_k)$. For simplicity, we give four representative examples belonging to each class:
\begin{itemize}
\item[(A)] $\xi(\lfloor b^k\rfloor)$ is irrational for every real number $b\geq 1+\sqrt{2}$;  
\item[(B)] $\xi((1+\sqrt{2})^k+(1-\sqrt{2})^k)$ is transcendental; 
\item[(C)] $\xi(r3^k-1)$ is transcendental for every integer $r\geq 4.003\times 10^{14}$;
\item[(D)] $\xi(3^{k-\lfloor (\log k)^{1/2} \rfloor}2^{\lfloor (\log k)^{1/2}\rfloor})$ is transcendental. 
\end{itemize}
\end{abstract} 

\maketitle

\section{Introduction}\label{Section:Intro}
Let $\mathbb{N}$ be the set of all positive integers. For all real numbers $x$, we define $\lfloor x\rfloor$ as the greatest integer less than or equal to $x$. In 1947, Mills \cite{Mills} constructed a real number $A>1$ such that 
\begin{equation}\label{Condition:Mills}
\text{$\lfloor A^{3^k} \rfloor$ is a prime number for all $k\in \mathbb{N}$}. 
\end{equation}
After that, mathematicians started to study variants and generalisations of Mills' result. For example, Wright \cite{Wright51} found a real number $\mu>1$ such that all of 
\[
\lfloor 2^\mu \rfloor, \quad \lfloor 2^{2^\mu} \rfloor, \quad  \lfloor 2^{2^{2^\mu}} \rfloor, \quad \cdots
\]
are prime numbers. We refer the readers to  \cite{Dudley} which is a good survey for quickly reviewing early works on prime-representing functions.  

The smallest real number $A$ satisfying \eqref{Condition:Mills} exists, and it is called Mills' constant (see \cite{Finch}). We will prove the existence of such a number in Lemma~\ref{Lemma:existence}. Assuming the Riemann Hypothesis (for short RH), Caldwell and Cheng \cite{CaldwellCheng} gave a decimal approximation to Mills' constant as follows: 
\[
1.30637\: 78838\: 63080\: 69046\: 86144\: 92602\cdots.
\]
Proving the irrationality of this number turned out to elude some previous efforts by Alkauskas and Dubickas \cite{AlkauskasDubickas} and by the author and Takeda \cite{SaitoTakeda}. However, the author \cite[Theorem~1.2]{Saito25} recently resolved that Mills' constant is irrational. Furthermore, the author \cite[Proposition~5.2]{Saito25} gave a sufficient condition to obtain the transcendency as follows.   
\begin{proposition}\label{Proposition:sufficient}
Assume that there exists a real number $\theta\in(0,1/2)$ such that for every sufficiently large real number $x$, we find a prime number $p$ satisfying $x\leq p \leq x+x^\theta$. Then, Mills' constant is transcendental.   
\end{proposition}

Under RH, Carneiro, Milinovich, and Soundararajan \cite[Theorem~1.5]{CMS} showed that  for all real numbers $x\geq 4$, there exists a prime number $p$ such that 
\begin{equation}\label{Equation:CMS}
p\in \left[x,x+\frac{22}{25}\sqrt{x}\log x\right]. 
\end{equation}
Therefore, even if RH is true, Proposition~\ref{Proposition:sufficient} cannot imply the transcendency of Mills' constant. Thus, in \cite[Question~5.3]{Saito25}, the author asked the following question. 

\begin{question}\label{Question:under_RH}
Assuming the Riemann hypothesis, is Mills' constant transcendental?
\end{question}

In this paper, we reveal that the answer is ``YES''.  Let us prepare several notations to present more general and precise results. Let $(C_k)_{k\ge 1}$ be a positive real sequence. We always write $c_1=C_1$ and $c_{k+1}=C_{k+1}/C_k$ for all $k\in \mathbb{N}$. Therefore, we have
\[
C_k=c_1 c_2\cdots c_k
\]
for all $k\in \mathbb{N}$. Let $\mathcal{P}$ be the set of all prime numbers. We define 
\[
\mathcal{W}(C_k)=\mathcal{W}((C_k)_{k\ge 1})=\{A>1\colon \lfloor A^{C_k} \rfloor\in \mathcal{P} \text{  for every $k\in \mathbb{N}$}   \}.
\]
For simplicity, we usually use the notation $\mathcal{W}(C_k)$ instead of $\mathcal{W}((C_k)_{k\ge 1})$, and so we remark that $\mathcal{W}(C_k)$ is independent of $k$. For instance, if $(C_k)_{k\ge 1}=(3^k)_{k\ge 1}$, then
\[
\mathcal{W}(3^k)= \{A >1\colon \lfloor A^{3^k} \rfloor\in \mathcal{P} \text{  for every $k\in \mathbb{N}$}\}. 
\]
By \cite[Lemma~3.1]{Saito25},  if $\mathcal{W}(C_k)$ is non-empty, then the smallest element of $\mathcal{W}(C_k)$ exists, say $\xi(C_k)$. Essentially by Matom\"{a}ki \cite[Theorem~3]{Matomaki}, we have the sufficient condition that $\mathcal{W}(C_k)$ is non-empty (\textit{i.e.} $\xi(C_k)$ exists) as follows.

\begin{theorem}\label{Theorem:Matomaki}
Let $(C_k)_{k\ge 1}$ be a positive real sequence satisfying 
\begin{enumerate}\renewcommand{\theenumi}{\arabic{enumi}}
\renewcommand{\labelenumi}{(\theenumi)} 
\item \label{Matomaki1} $c_1>0$\textup{;} 
\item\label{Matomaki2} $c_{k+1}\geq 2$ for all $k\in \mathbb{N}$.
\end{enumerate}
Then, $\mathcal{W}(C_k)$ is non-empty and $\xi(C_k)$ exists.
\end{theorem}
 
\begin{remark}
Actually, Matom\"{a}ki showed that $\mathcal{W}(C_k)$ is uncountable, nowhere dense, and has Lebesgue measure $0$ if a sequence $(c_k)_{k\ge 1}$ of real numbers satisfies $c_1\geq 2$ and \eqref{Matomaki2}. There is no essential difference between the initial conditions $c_1>0 $ and $c_1\geq 2$ (see \cite[Remark~3.4]{SaitoTakeda}). 
\end{remark} 
 
It is natural to ask whether $\xi(C_k)$ is irrational or not. If the number is irrational, then we should further ask whether it is transcendental or not. A Pisot number plays a key role in answering the question.  We say that a real algebraic integer $\beta>1$ is a \textit{Pisot number} if all conjugates of $\beta$ over $\mathbb{Q}$ lie in the open unit disk of $\mathbb{C}$ except for itself. An algebraic number is  \textit{quadratic} or \textit{cubic} if its degree is $2$ or $3$, respectively.  An algebraic number is \textit{totally real} if all its Galois conjugates over $\mathbb{Q}$ are real. 

For example, the golden ratio is a totally real quadratic Pisot number. The unique real root of $X^3-X-1$ is a cubic Pisot number. In \cite[Theorem~1.5]{Saito25}, the author proved the following result. 

\begin{theorem}\label{Theorem:Previous}
Let $(C_k)_{k\ge 1}$ be a positive real sequence satisfying  \eqref{Matomaki1}, \eqref{Matomaki2} in Theorem~\ref{Theorem:Matomaki}, and $c_{k}\in \mathbb{N}$ for all $k\in \mathbb{N}$. Then, the smallest real number of $\mathcal{W}(C_k)$ exists, say $\xi$. If $\limsup_{k\to \infty} c_{k+1}\geq 4$, then $\xi$ is transcendental.  If $\limsup_{k\to \infty} c_{k+1}=3$, then either $\xi$ is transcendental, or there exists $m\in \mathbb{N}$ such that  $\xi^{C_m}$ is a cubic Pisot number. 
\end{theorem}

This theorem implies that $\xi(3^k)$ is irrational, that is, Mills' constant is irrational.  In addition, $\xi(c^k)$ is transcendental for all integers $c\geq 4$. In this paper, we first improve the theorem in the case when $c_{k+1}$ is eventually equal to $3$. To assert our main result, we consider the following condition \eqref{dag} for a real positive number $\theta$: there are real numbers $x_0>0$ and $d>0$ such that for every $x\ge x_0$, we have
\begin{equation}\label{dag} \tag{$\dag$}
\# (\mathcal{P} \cap [x, x+x^\theta]) \geq \frac{dx^\theta}{\log x}.
\end{equation} 
Baker, Harman, and Pintz \cite{BakerHarmanPintz} provided the best-known result on \eqref{dag}. 
\begin{theorem}\label{Theorem:BHP}
The condition \eqref{dag} is true for $\theta=21/40$.
\end{theorem}

By applying symbolic dynamics, we obtain the following main result in this paper.

\begin{theorem}\label{Theorem:Main5}  Let $(C_k)_{k\ge 1}$ be a positive real sequence satisfying  \eqref{Matomaki1}, \eqref{Matomaki2} in Theorem~\ref{Theorem:Matomaki}, and $c_{k}\in \mathbb{N}$ for all $k\in \mathbb{N}$. Then, the smallest real number of $\mathcal{W}(C_k)$ exists, say $\xi$. If $c_{k+1}=3$ for every sufficiently large $k\in \mathbb{N}$, then either $\xi$ is transcendental, or there exists $m\in \mathbb{N}$ such that  $(\beta\coloneqq) \xi^{C_m}$ is a totally real cubic Pisot number. If the latter is true, then  we have 
\begin{equation}\label{Inequality:Theorem:Main5}
|\beta_3|< -\beta_2\leq \min (|\beta_3|^{17/23}, \beta^{-17/40}), 
\end{equation}
where $\beta_2$ and $\beta_3$ denote the real conjugates of $\beta$ satisfying $\beta>1>|\beta_2|>|\beta_3|$. Furthermore, if \eqref{dag} is true for every fixed $\theta>1/2$, then $\xi$ is transcendental.
\end{theorem}

The condition \eqref{Inequality:Theorem:Main5} seems to give more restrictions on the potential of Pisot numbers, but there are infinitely many totally real cubic Pisot numbers $\beta$ satisfying  \eqref{Inequality:Theorem:Main5} by using Dubickas' result \cite[Theorem]{Dubickas2004} with $d\coloneqq 3$, $\lambda_1\coloneqq -s^{-17/40}/2$, $\lambda_2\coloneqq s^{-17/40}/4$, and $\epsilon\coloneqq s^{-17/40}/8$, where $s$ runs over sufficiently large even numbers. 

Let $N(\sigma, T)$ be the number of zeroes $\rho=\beta+i\gamma$ of the Riemann zeta function in the rectangle $\sigma \leq \beta\leq 1$ and $0<\gamma\leq T$. It is conjectured that for every fixed $\epsilon>0$, we have 
\[
N(\sigma, T) \ll_\epsilon T^{2(1-\sigma)+\epsilon }
\]
uniformly for $1/2\leq  \sigma \leq 1$ as $T\to \infty$. The conjecture is called the \textit{Density Hypothesis} (for short DH). By Ingham's result \cite[Theorem~1]{Ingham}, if DH is true, then 
\[
\lim_{x\to \infty}  \frac{\# (\mathcal{P} \cap [x, x+x^\theta])}{x^\theta/\log x }=1 
\]
for any fixed $\theta>1/2$. We note that RH implies DH. Therefore, Theorem~\ref{Theorem:Main5} leads to the affirmative answer to Question~\ref{Question:under_RH}. 

\begin{theorem}
If DH is true \textup{(}more strongly if RH is true\textup{)}, then Mills' constant is transcendental. 
\end{theorem}

In the case when $\limsup_{k\to \infty} c_{k+1}\leq 3$,  can we prove that $\xi(C_k)$ is transcendental without any hypothesis on the Riemann zeta function?  In this paper, we also aim to obtain several classes of sequences $(C_k)_{k\ge 1}$ satisfying that $\limsup_{k\to \infty} c_{k+1}$ is small, but we can verify the arithmetic properties of $\xi(C_k)$. For simplicity, we propose four representative examples belonging to each class. We will give more general results in the next section.

\begin{theorem} \label{Theorem:main:example}
 The following statements are true:
 \begin{enumerate}\renewcommand{\theenumi}{\Alph{enumi}}
\renewcommand{\labelenumi}{(\theenumi)} 
\item \label{TypeA}$\xi(\lfloor b^k\rfloor)$ is irrational for every fixed real number $b\geq 1+\sqrt{2}$\textup{;}
\item \label{TypeB}$\xi((1+\sqrt{2})^k+(1-\sqrt{2})^k)$ is transcendetal\textup{;}
\item \label{TypeC}$\xi(r3^k-1)$ is transcendental for every fixed integer $r\geq 4.003\times 10^{14}$\textup{;}
\item \label{TypeD}$\xi(3^{k- \lfloor \log k \rfloor}2^{\lfloor \log k \rfloor} )$ is transcendental.  
\end{enumerate}
\end{theorem}

In \eqref{TypeA}, \eqref{TypeB}, and \eqref{TypeC}, we extend the condition $c_k=C_{k}/C_{k-1}\in \mathbb{N}$ to $c_k\in \mathbb{Q}\cap (0,\infty)$. 
Alkauskas and Dubickas \cite{AlkauskasDubickas} have already given such an extension. They constructed a specific transcendental number in $\mathcal{W}(C_k)$ if  $(c_k)_{k\ge 1}$ is a sequence of rational numbers satisfying $c_1=1$, $c_{k+1}>2.1053$ $(k=1,2,3,\ldots)$, $\limsup_{k\to \infty} c_{k+1}=\infty$, and $C_{k}=c_1\cdots c_k \in \mathbb{N}$. However, they did not discuss the case $\limsup_{k\to \infty} c_{k+1}<\infty$ or the arithmetic properties of the smallest element of $\mathcal{W}(C_k)$.

\begin{notation} For all $x\in \mathbb{R}$, $\{x\}$ denotes the fractional part of $x$ \textit{i.e.} $\{x\}=x-\lfloor x\rfloor$. For all finite sets $A$, $\# A$ denotes the number of elements in $A$.  We say that $f(x)\ll g(x)$ for all $x\in X$ if there exists $C>0$ such that $|f(x)| \leq C g(x)$ for $x\in X$. If $C$ depends on some parameters $y_1,\ldots, y_\ell$, we write $f(x)\ll_{y_1,\ldots, y_d} g(x)$.      
\end{notation}

\section{Types~A,B,C,D and Proof of Theorem~\ref{Theorem:main:example}}\label{Section:TypeABC}

In this section, we exhibit four classes of $(C_k)_{k\ge 1}$ which correspond to \eqref{TypeA}, \eqref{TypeB}, \eqref{TypeC}, \eqref{TypeD} in Theorem~\ref{Theorem:main:example}, say Type~A, B, C, and D. These names are symbolic, and so they have no deeper meaning.
\subsection{Type A} 

\begin{theorem}\label{Theorem:Main1}
Suppose that $(c_k)_{k\ge 1}$ is a real sequence satisfying 
\begin{enumerate}\renewcommand{\theenumi}{A\arabic{enumi}}
\renewcommand{\labelenumi}{(\theenumi)} 
\item \label{A1}$c_1\geq 1$\textup{;}
\item \label{A2}$c_{k+1}\geq 2$ for all $k\in \mathbb{N}$\textup{;}
\item \label{A3}$\limsup_{k\to \infty} c_{k+1}>40/19$\textup{;}
\item \label{A4}$C_k=c_1\cdots c_k\in \mathbb{N}$ for all $k\in \mathbb{N}$.
\end{enumerate}
Then, $\xi=\xi(C_k)$ exists. Furthermore, either $\xi$ is transcendental, or there exists $m\in \mathbb{N}$ such that $\xi^{C_m}$ is a Pisot number of degree $\ell$ satisfying 
\begin{equation}\label{equation:TypeA:1}
1\leq \ell \leq \displaystyle{1+ \left(\frac{19}{40}\limsup_{k\to \infty} c_{k+1} -1 \right)^{-1}}.
\end{equation}
In particular, $\xi$ is irrational.
\end{theorem} 
We will prove Theorem~\ref{Theorem:Main1} in Section~\ref{Section:TypeA}.  A feature of Type~A is that it does not assume algebraic conditions on $(C_k)_{k\ge 1}$ such as divisibility and congruence conditions.   This theorem implies the irrationality of $\xi(\lfloor \alpha b^k  \rfloor)$, although the algebraic properties of $\lfloor \alpha b^k  \rfloor$ are highly unclear for given specific real numbers $b$ and $\alpha$. 
 
\begin{corollary}\label{Corollary:b-adic}Let $b$ and $\alpha$ be real numbers with $b>40/19=2.10526\cdots$ and $\alpha\geq \max (\frac{1}{b(b-2)}, \frac{1}{b})$. Then, $\xi=\xi(\lfloor \alpha b^{k}\rfloor )$ exists. Furthermore, either $\xi$ is transcendental, or there exists $k\in \mathbb{N}$ such that $\xi^{\lfloor \alpha b^k\rfloor}$ is a Pisot number of degree $\ell$ satisfying
\[
1\leq \ell \leq 1+ \left(\frac{19}{40}b-1 \right)^{-1}
\]
In particular,  $\xi$ is irrational.
\end{corollary}

\begin{proof} Let $C_k=\lfloor \alpha b^k\rfloor$ for every $k\in \mathbb{N}$. We recall that $c_1=C_1$, and $c_{k+1}=C_{k+1}/C_k$ for every $k\in \mathbb{N}$. Then, it is clear that $(c_k)_{k\ge 1}$ satisfies \eqref{A3} and \eqref{A4} in Theorem~\ref{Theorem:Main1} since $\lim_{k\to \infty} c_{k+1}=b>40/19$ and $C_k=\lfloor \alpha b^k\rfloor\in \mathbb{N}$ for all $k\in \mathbb{N}$. Furthermore, by $\alpha\geq \max (\frac{1}{b(b-2)}, \frac{1}{b})$, we have 
\begin{gather*}
c_1=C_1=\lfloor \alpha b\rfloor\geq \lfloor (1/b) \cdot b \rfloor=1,\\
c_{k+1}=\frac{C_{k+1}}{C_k}=\frac{\lfloor \alpha b^{k+1}\rfloor }{\lfloor \alpha b^k \rfloor} \geq \frac{\alpha b^{k+1}-1}{\alpha b^k}  \geq  b- \frac{1}{\alpha b} \geq b-\frac{1}{\frac{b}{b(b-2)} } =2 
\end{gather*}
for all $k\in \mathbb{N}$. Thus, $(c_k)_{k\ge 1}$ satisfies \eqref{A1} and \eqref{A2}. Therefore, Theorem~\ref{Theorem:Main1} implies Corollary~\ref{Corollary:b-adic}. 
\end{proof}
 
\begin{proof}[Proof of \eqref{TypeA} in Theorem~\ref{Theorem:main:example}] For all real numbers $b>2$, we observe that 
\[
1\geq \frac{1}{b(b-2)} \iff b^2-2b-1\geq 0 \iff  b\geq 1+\sqrt{2}.  
\]
Applying Corollary~\ref{Corollary:b-adic} with $\alpha=1$ and $b\geq  1+\sqrt{2}$, we obtain \eqref{TypeA}. 
\end{proof}

\subsection{Type B}

\begin{theorem}\label{Theorem:Main2}
Suppose that $(c_k)_{k\ge 1}$ is a real sequence satisfying 
\begin{enumerate}\renewcommand{\theenumi}{B\arabic{enumi}}
\renewcommand{\labelenumi}{(\theenumi)} 

\item \label{B1}$c_1\geq 1$\textup{;}
\item \label{B2}$c_{k+1}\geq 2$ for all $k\in \mathbb{N}$\textup{;}
\item \label{B3}$\limsup_{k\to \infty} c_{k+1}>40/19$\textup{;}
\item \label{B4}$C_k=c_1\cdots c_k\in \mathbb{N}$ for all $k\in \mathbb{N}$.
\end{enumerate}
Then, $\xi=\xi(C_k)$ exists. Let $\epsilon$ be a sufficiently small positive real number such that $I\coloneqq \{k\in \mathbb{N}\colon c_{k+1}\geq 40/19 +\epsilon\}$ is infinite.  We further suppose that 
\begin{enumerate}
\setcounter{enumi}{4}
\renewcommand{\theenumi}{B\arabic{enumi}}
\renewcommand{\labelenumi}{(\theenumi)} 
\item \label{B5}for all $m\in \mathbb{N}$ there exists $k\in I$ with $k>m$ such that $C_m\mid C_k$.
\end{enumerate} 
Then, either $\xi$ is transcendental, or there exists $g\in \mathbb{N}$ such that $\xi^{g}$ is a Pisot number of degree $\ell$ satisfying
\[
3\leq \ell \leq 1+ \left(\frac{19}{40}\limsup_{k\to \infty} c_{k+1} -1 \right)^{-1},
\]  
and for every sufficiently large $k\in I$, we have $g\mid C_k$. In addition, we suppose that 
\begin{enumerate}
\setcounter{enumi}{5}\renewcommand{\theenumi}{B\arabic{enumi}}
\renewcommand{\labelenumi}{(\theenumi)} 
\item \label{B6} for all $m\in \mathbb{N}$ and $L\in \mathbb{N}$ there exists $k\in I$ with $k>m$ such that 
\[
C_k \equiv C_m \mod LC_m.
\]
\end{enumerate}
Then, $\xi$ is transcendental.  
\end{theorem}

\begin{remark}
This theorem implies Theorem~\ref{Theorem:Previous}. Therefore, Theorem~\ref{Theorem:Main2} is an extension of Theorem~\ref{Theorem:Previous} that is the previous main result in \cite{Saito25}.
\end{remark}

We will prove Theorem~\ref{Theorem:Main2} in Section~\ref{Section:TypeB}. We note that \eqref{B1} to \eqref{B4} are coincedent with \eqref{A1} to \eqref{A4}. Thus, by adding the algebraic conditions \eqref{B5} and \eqref{B6}, we can find stronger arithmetic properties of $\xi(C_k)$. A certain class of linear recurrence sequences satisfies \eqref{B1} to \eqref{B6}. 

\begin{corollary}\label{Corollary:Recurrence} 
Let $d$ be an integer greater than or equal to $2$. Let $a_1,\ldots, a_{d-1}$ be integers, and let  $a_0\in \{-1,1\}$.  Let $(R_k)_{k\ge 1}$ be a sequence of positive integers satisfying 
\begin{enumerate}\renewcommand{\theenumi}{B\arabic{enumi}'}
\renewcommand{\labelenumi}{(\theenumi)} 
\item \label{B1'}$R_1\geq 1$\textup{;}
\item \label{B2'}$R_{k+1}/R_k\geq 2$ for all $k\in \mathbb{N}$\textup{;}
\item \label{B3'}$\liminf_{k\to \infty} R_{k+1}/R_k> 40/19$\textup{;}
\item \label{B4'}$R_{k+d}= a_{d-1}R_{k+d-1}+a_{d-2} R_{k+d-2}+ \cdots + a_1 R_{k+1}+a_0R_{k}$\quad  for all $k\in \mathbb{N}$.
\end{enumerate}
Then, the number $\xi(R_k)$ exists, and it is transcendental.
\end{corollary}

We will prove Corollary~\ref{Corollary:Recurrence} in Section~\ref{Section:TypeB}. 

\begin{proof}[Proof of \eqref{TypeB} in Theorem~\ref{Theorem:main:example}]
 Let $R_k=(1+\sqrt{2})^k +(1-\sqrt{2})^k$ for all $k\in \mathbb{N}$. Then, by basic calculation, we have $R_1=2$, $R_2=4$, and 
 \[
 R_{k+2}=2R_{k+1}+R_k
 \]
  for all $k\in \mathbb{N}$. Therefore, $(R_k)_{k\ge 1}$ satisfies \eqref{B1'} to \eqref{B4'}. By applying Corollary~\ref{Corollary:Recurrence}, $\xi(R_k)=\xi( (1+\sqrt{2})^k +(1-\sqrt{2})^k)$ is transcendental. 
\end{proof}

\subsection{Type~C} 

For all integers $a_1,a_2,\ldots, a_k$, we define $\gcd(a_1,\ldots a_k)$ as the gratest positive integers $d$ such that $d \mid a_1$, $d\mid a_2$, $\ldots$ , $d \mid a_k$. For all infinite sequences $(a_k)_{k\ge 1}$ of integers, $\gcd (a_1, a_{2},\ldots, a_k)$ is decreasing for $k=1,2,\ldots$, and hence we define 
\[
\gcd (a_1, a_{2},\ldots)=\lim_{k\to \infty } \gcd (a_1, a_{2},\ldots, a_k),
\] 
which is well-defined. We also see that $\gcd (a_m, a_{m+1},\ldots)$ is increasing for $m=1,2,\ldots$, and so we define the \textit{asymptotic greatest common divisor of} $(a_k)_{k\ge 1}$ by
\[
\lim_{m\to \infty}\gcd (a_m, a_{m+1},\ldots)\in \mathbb{N} \cup \{\infty\}
\]
and it is denoted by $\mathrm{agcd}((a_k)_{k\ge 1})=\mathrm{agcd}(a_k)$.

\begin{theorem}\label{Theorem:Main3}Let $c$ be a real number grater than or equal to $3$. Assume that there exists $x_0\geq 1$ such that for all real numbers $x\geq x_0$,  
\begin{equation}\label{P1}
\mathcal{P} \cap [x, x+cx^{1-1/c}]\neq \emptyset. 
\end{equation}
Suppose that $(c_k)_{k\ge 1}$ is a sequence of real numbers satisfying
\renewcommand{\theenumi}{C\arabic{enumi}}
\renewcommand{\labelenumi}{(\theenumi)}   
\begin{enumerate}
\item \label{C1}$c_1\geq 1$\textup{;}
\item \label{C2}$c_{k+1}\geq c$ for all $k\in \mathbb{N}$\textup{;}
\item \label{C3}$C_k=c_1\cdots c_k \in \mathbb{N}$ for all $k\in \mathbb{N}$\textup{;}
\item \label{C4}for all $m\in \mathbb{N}$ there exists $k\in \mathbb{N}$ such that $C_m \mid C_k$\textup{;}
\item \label{C5}$\mathrm{agcd} (C_k) <\infty$. 
\end{enumerate}
Then, $\xi=\xi(C_k)$ exists. Furthermore, either $\xi$ is transcendental or there exists $g\in \mathbb{N}$ such that $g\mid \mathrm{agcd} (C_k)$ and $\xi^g$ is a cubic Pisot number less than or equal to
\[
 (2x_0^{1/c_2}+1)^{g/c_1}.
\]
\end{theorem}

In the case $c=3$, by applying \cite[Theorem~1.1]{MTY} and \cite[Section~4]{Cully-Hugill}, Mossinghoff, Trudgian, and Yang \cite[Section~2]{MTY} showed that \eqref{P1} is true for all 
\[
x\geq \exp(3\exp(32.76)).
\]

\begin{proof}[Proof of \eqref{TypeC} in Theorem~\ref{Theorem:main:example}]
Fix an arbitrary integer $r\ge 1$. Let $C_k=r 3^k-1$ for every $k\in \mathbb{N}$. Then, \eqref{C1} and \eqref{C3} are clearly valid. Let us verify \eqref{C2}, \eqref{C4}, and \eqref{C5}.

  For \eqref{C2}, we obtain $C_{k+1}=r 3^{k+1}-1 = 3(r3^k-1)+2\geq 3 C_{k}$. 

For \eqref{C4}, it follows that $\gcd(r,C_m)=1$ for every $m\in \mathbb{N}$ since $C_m=r3^m-1$. Thus, 
\[
3^m\equiv r^{-1} \mod C_m,
\] 
where $r^{-1}$ denotes the inverse of $r$ in $\mathbb{Z}/C_m \mathbb{Z}$. Let $\varphi(n)$ be the number of positive integers up to $n$ which is coprime to $n$. Euler's theorem implies that 
\[
r 3^{m\varphi(C_m)+m}-1= r(3^m)^{\varphi(C_m)+1 }-1 \equiv r \cdot r^{-1} -1 \equiv 0 \mod C_m.
\]
Therefore, we have $C_m\mid C_k$ for $k=m\varphi(C_m)+m$, which is \eqref{C4}. 

For \eqref{C5}, it suffices to show that  
\begin{equation}\label{equation:typeC:1}
\mathrm{agcd}(C_k) =\left\{ \begin{aligned}
 2\quad & \text{if $r$ is odd,}\\
1 \quad   & \text{if $r$ is even}.   
\end{aligned}\right. 
\end{equation}
 We now take an arbitrary prime factor $p$ of $C_m=r3^m-1$. Then, $3^m\equiv r^{-1} \mod p$ since  $\gcd(r,p)=1$. Therefore, we have
\begin{equation}\label{equation:typeC:2}
C_{m+1}=r3^{m+1 }-1 =r (3^m)\cdot 3 -1  \equiv 3-1\equiv 2 \mod p.
\end{equation}
If $r$ is even, then we have $\gcd(C_m,C_{m+1})=1$ for every $m\in \mathbb{N}$, and hence $\mathrm{agcd}(C_k)=1$.  

Suppose that $r$ is odd. Then, by \eqref{equation:typeC:2}, it is clear that the prime factor of $\gcd(C_m,C_{m+1})$ is only $2$. Furthermore, if $r\equiv 1 \mod 4$, then 
\[
C_{2k+1}=r3^{2k+1}-1\equiv (-1)^{2k+1}-1 \equiv -2 \not\equiv 0\mod 4
\]
for all $k\in \mathbb{N}$. If $r\equiv 3 \mod 4$, then 
\[
C_{2k}=r3^{2k}-1\equiv 3(-1)^{2k}-1 \equiv 2\not\equiv 0 \mod 4
\]
for all $k\in \mathbb{N}$. Therefore, $\gcd(C_m, C_{m+1})=2$ for all $m\in \mathbb{N}$, and hence we have  \eqref{equation:typeC:1} and \eqref{C5}.  

By applying Theorem~\ref{Theorem:Main3}, $\xi=\xi(r 3^k-1)$ exists, and either $\xi$ is transcendental or there exists $g\in \{1,2\}$ such that $\xi^g$ is a cubic Pisot number satisfying 
\[
\xi^g \leq (2x_0^{1/c_2} +1 )^{g/c_1}\leq (2x_0^{1/3} +1 )^{2/(3r-1) }.
\]
Let $\kappa$ be the unique real root of $X^3-X-1$, which is approximately equal to 1.324717$\cdots$. Then by Siegel's result \cite{Siegel}, $\kappa$ is the smallest Pisot number. Therefore, if we have
 \[
 (2x_0^{1/3} +1 )^{2/(3r-1)}< \kappa\quad \iff \quad  \frac{2\log (2x_0^{1/3} +1 )}{3\log\kappa} +\frac{1}{3}< r ,
\]
then $\xi$ is transcendental. Choosing $x_0=\exp(3\exp(32.76))$ by  \cite[Section~2]{MTY} and using Mathematica, we obtain
\[
\frac{2\log (2x_0^{1/3} +1 )}{3\log\kappa} +\frac{1}{3}=400296054181891.5\cdots.  \qedhere
\] 

\end{proof}

It is natural to expect that  \eqref{P1} with $c=3$ holds for all $x\ge 1$. If it were true, then we could replace $4.003\times 10^{14}$ with $1$ as follows.
\begin{corollary}\label{Corollary:AssumeRH}
Assume that \eqref{P1} with $c=3$ is true for all real numbers $x\geq 1$. Then, for every fixed integer $r\geq 1$, the number $\xi(r3^k-1)$ exists, and it is transcendental. 
\end{corollary}

We will prove Corollary~\ref{Corollary:AssumeRH} in Appendix~\ref{AppendixA}. Assuming RH,  \eqref{Equation:CMS} implies that \eqref{P1} with $c=3$ is true for all real numbers $x\geq 1$. As a partial result, Cully-Hugill \cite{Cully-Hugill} proved that \eqref{P1} with $c=155$ is true for all $x\ge 1$. More recently, she and Johnston  \cite{CullyHugillJohnston1,CullyHugillJohnston2} proved that  \eqref{P1} with $c=90$ is true for all $x\ge 1$.

\subsection{Type~D} Let us focus on a sequence $(c_k)_{k\ge 1}$ satisfying $c_k\in \{2,3\}$ for every sufficiently large $k\ge 1$. As long as there is no confusion, we write $a_1a_2\cdots a_n$ as $(a_1,\ldots,a_n)$. Further, for every $\nu\in \mathbb{Z}_{\ge 0}\cup \{\infty\}$, we define 
\[
3^\nu=(\underbrace{3,\ldots, 3}_{\nu}),
\]
 where $3^0=\emptyset$ and $3^\infty = (3,3,3,\ldots)$. For example, 
\[
2\ 3^4\ 2\ 3^0\ 2\ 3^\infty =(2,3,3,3,3,2,2,3,3,3,\ldots).
\]
If a sequence $(c_k)_{k\ge 1}$ satisfies $c_k\in \{2,3\}$ for every sufficiently large $k$, then there exist $k_0>0$ and $\nu_1,\nu_2,\ldots\in \mathbb{Z}_{\ge 0} \cup \{\infty\}$ such that
\begin{equation}\label{Definition:nu}
(c_k)_{k\ge k_0}= 3^{\nu_1}\ 2\ 3^{\nu_2}\ 2\ 3^{\nu_3}\ \cdots.
\end{equation}

\begin{theorem}\label{Theorem:Main4} Suppose that $(c_k)_{k\ge 1}$ is a sequence of positive integers satisfying 
\renewcommand{\theenumi}{D\arabic{enumi}}
\renewcommand{\labelenumi}{(\theenumi)}   
\begin{enumerate}
\item \label{D1}$c_{k+1}\geq  2$ for all $k\in \mathbb{N}$\textup{;}
\item \label{D2}there exsits $k_0>0$ such that $c_{k+1}\in \{2,3\}$ for all $k\ge k_0$\textup{;} 
\item \label{D3}$c_{k}=2$ and $c_{k+1}=3$ for infinitely many $k$\textup{;}
\item \label{D4}for $(\nu_j)_{j\ge 1}$ defined in \eqref{Definition:nu}, we have
\begin{equation}\label{Condtion:D4}
\limsup_{N\to \infty} \frac{\nu_{N+1}}{\sum_{j=1}^N\nu_{j} + \frac{\log 2}{\log 3} N}=\infty. 
\end{equation}
\end{enumerate}
Then, $\xi(C_k)$ exists, and it is transcendental. 
\end{theorem}

We note that the constant $\log 2/\log 3$ in \eqref{Condtion:D4} is natural since 
\[
3^{\sum_{j=1}^N\nu_{j} + \frac{\log 2}{\log 3} N}=2\times 3^{\nu_1}\times2\times 3^{\nu_2} \times \cdots \times 2\times 3^{\nu_N }
\]
and we will apply this relation in the proof of Theorem~\ref{Theorem:Main4}.
\begin{remark}
The sequence $(\nu_j)_{j\ge 1}$ depends on the choice of $k_0$, but the left-hand side of \eqref{Condtion:D4} does not depend on  $k_0$. 
\end{remark}

\begin{proof}[Proof of \eqref{TypeD} in Theorem~\ref{Theorem:main:example}]
Let $C_k=3^{k-\lfloor (\log k)^{1/2}\rfloor} 2^{\lfloor (\log k)^{1/2} \rfloor}$ for every $k\in \mathbb{N}$. Then, we observe that $c_1=3$ and for every $k\in \mathbb{N}$
\begin{equation}
c_{k+1}=\frac{C_{k+1}}{C_k}= \begin{cases}
3 & \text{if } e^{j^2}\leq k <  e^{(j+1)^2}-1 \text{ for some $j\in \mathbb{Z}_{\ge 0}$,}\\
2 & \text{if }  e^{(j+1)^2}-1 \leq k < e^{(j+1)^2}  \text{ for some $j\in \mathbb{Z}_{\ge 0}$.}
\end{cases} 
\end{equation}
Therefore, $(\nu_j)_{j\ge 1}$ defined in \eqref{Definition:nu} satisfies $\nu_j= e^{(j+1)^2}-e^{j^2} +O(1)$, and hence  
\[
\frac{\nu_{N+1}}{\sum_{j=1}^N\nu_{j} + \frac{\log 2}{\log 3} N} \gg \frac{e^{(N+2)^2}-e^{(N+1)^2} +O(1)  }{ e^{(N+1)^2} +O(N)  }\gg e^{(N+2)^2-(N+1)^2} \to \infty
\]
as $N\to \infty$. Therefore, by Theorem~\ref{Theorem:Main4}, we obtain \eqref{TypeD} in Theorem~\ref{Theorem:main:example}.
\end{proof}

\section{A general result}

In this section, we present a key proposition which implies Theorems~\ref{Theorem:Main1}, \ref{Theorem:Main2}, and \ref{Theorem:Main3}.  Before claiming, let $\Tr(\beta)$ be the trace of $\beta$ over $\mathbb{Q}$ for every algebraic number $\beta$, that is, $\Tr(\beta)=\beta_1+\cdots+\beta_\ell$, where $\beta_1,\ldots,\beta_\ell$ denote all the conjugates of $\beta$ over $\mathbb{Q}$. 

\begin{proposition}\label{Proposition:General}
Let $\theta$ be a real number in $[1/2,1)$ satisfying \eqref{dag}. Suppose that $(c_k)_{k\ge 1}$ is a sequence of real numbers satisfying 
\begin{enumerate}\renewcommand{\theenumi}{G\arabic{enumi}}
\renewcommand{\labelenumi}{(\theenumi)} 
\item \label{G1}$c_1\ge 1$\textup{;} 
\item \label{G2}$c_{k+1}\geq 2$ for all $k\in \mathbb{N}$\textup{;} 
\item \label{G3}$\limsup_{k\to\infty}c_{k+1}>  1/(1-\theta)$\textup{;}
\item \label{G4}$C_k=c_1\cdots c_k\in \mathbb{N}$ for every sufficiently large $k\in \mathbb{N}$.
\end{enumerate}
Then, $\xi=\xi(C_k)$ exists. We further suppose that for every  $m\in \mathbb{N}$, we have
\begin{equation}\label{ineq:key1-assump1}
 \xi^{C_m}\notin \mathbb{N},
\end{equation}
and we suppose that $\xi$ is algebraic. Let $\epsilon$ be a fixed small real number with 
\begin{equation}\label{Condition:epsilon}
0<\epsilon<\limsup_{k\to\infty}c_{k+1}-1/(1-\theta). 
\end{equation}
Let $I=\{k\in \mathbb{N} \colon c_{k+1}\geq 1/(1-\theta)+\epsilon\}$. Then, the following properties are true\textup{:}
\begin{enumerate}\renewcommand{\theenumi}{\roman{enumi}}
\renewcommand{\labelenumi}{(\theenumi)}
\item \label{Result:General:1}We have 
\[
\{ \xi^{C_k} \} \leq \frac{2}{c_{k+1}\lfloor \xi^{C_k}\rfloor^{(1-\theta)c_{k+1}-1}  }
\]
for every sufficiently large $k\in I$\textup{;} 
\item \label{Result:General:2}There exists $g\in \mathbb{N}$ such that  $\xi^g=\beta$ is a Pisot number of degree $\ell$ with 
\[
2\leq \ell \leq 1+ \biggl((1-\theta)\: \underset{\substack{k\to \infty}}{\limsup}\: c_{k+1}-1\biggl)^{-1}
\]
and we have $g \mid C_k$ for every sufficiently large $k\in I$\textup{;} 
\item \label{Result:General:3} If $\ell=2$, then $\xi^g=\beta=(1+\sqrt{5})/2$ and $C_k/g$ is an odd prime number for every sufficiently large $k\in I$\textup{;} 
\item \label{Result:General:4} We have $\Tr(\beta^{C_{k}/g }) =\lfloor \xi^{C_{k}} \rfloor$ for every sufficiently large $k\in I$. 
\end{enumerate}
\end{proposition}

The rest of the paper is organised as follows. In Section~\ref{Section:Existence}, we quickly review the proof of Theorem~\ref{Theorem:Matomaki}. In Section~\ref{Section:Lemmas}, we prepare auxiliary results for proving Proposition~\ref{Proposition:General}. In Section~\ref{Section:ProofGenProp}, we complete the proof of Proposition~\ref{Proposition:General}. In Sections~\ref{Section:TypeA}, \ref{Section:TypeB}, and \ref{Section:TypeC},  we prove Theorems~\ref{Theorem:Main1}, \ref{Theorem:Main2}, and \ref{Theorem:Main3}, respectively. At last, we prove Theorems~\ref{Theorem:Main5} and \ref{Theorem:Main4} in Section~\ref{Section:TypeD}

\section{Proof of Theorem~\ref{Theorem:Matomaki}}\label{Section:Existence}

\begin{lemma}[{\cite[Lemma~4.1]{SaitoTakeda} and \cite[Lemma~3.1]{Saito25}}]\label{Lemma:existence}
Let $(C_k)_{k\ge 1}$ be a sequence of positive real numbers. If $\mathcal{W}(C_k)$ is non-empty, then the smallest element of $\mathcal{W}(C_k)$ exists. 
\end{lemma}
\begin{proof}
Since $\mathcal{W}(C_k)$ is non-empty and lower bounded,  its infimum exists, say $\zeta$. Let $k$ be a fixed positive integer. By the definition of the infimum, we find a sequence $A_1>A_2>\cdots$ of $\mathcal{W}(C_k)$ satisfying $\lim_{j\to \infty} A_j=\zeta$. By the right-hand side continuity of $\lfloor \cdot \rfloor$, we obtain  
\[
\lfloor \zeta^{C_k}\rfloor = \lfloor\lim_{j\to \infty}  A_j^{C_k}\rfloor = \lfloor A_j^{C_k}\rfloor\in \mathcal{P}. 
\]
 Therefore, $\zeta$ is the smallest element of $\mathcal{W}(C_k)$. 
\end{proof}

\begin{lemma}\label{Lemma-Matomaki3}
Let $(e_k)_{k\ge 2}$ be a real sequence with $e_k\geq 2$ for all $k\geq 2$. Let $\epsilon_0$ be a sufficiently small positive real number.  Then, there exist real numbers  $X_1=X_1(\epsilon_0)>0$ and $d_2>0$ such that for all real numbers $X$ and $\eta$ satisfying $X\geq X_1$ and $\eta\in [1/2,1-\epsilon_0]$, if we have
\begin{equation}\label{eq:NumberPrimes1}
\#([X,X+X^\eta] \cap \mathcal{P} ) \geq  \frac{d_2 X^\eta}{\log X}, 
\end{equation}
then we can find a sequence  $(q_{k})_{k\ge 1}$ of prime numbers such that $X\leq q_1\leq  X+X^\eta$ and for all $k\in \mathbb{N}$, we have  
\begin{equation}\label{Inequality-Matomaki}
q_{k}^{e_{k+1}} \leq q_{k+1}< (q_{k}+1)^{e_{k+1}}-1 .   
\end{equation}
\end{lemma}

\begin{proof}
This lemma is deduced from \cite[Lemma~3.2]{SaitoTakeda} with $X\coloneqq X$ and $Y\coloneqq X+X^\eta$.  We note that it essentially follows from Matom\"{a}ki's results  \cite[Lemma~9]{Matomaki} and  \cite[Lemma~1.2]{Matomaki2007}. 
\end{proof}

\begin{proof}[Proof of Theorem~\ref{Theorem:Matomaki}] Let $\epsilon_0$ be a positive real number, and let $\eta\in[1/2,1-\epsilon_0]$. Let $X_1=X_1(\epsilon)$ be as in Lemma~\ref{Lemma-Matomaki3}.  Let $Y_0$ be a sufficiently large real number with $Y_0\geq X_1$, and we take $Y\geq Y_0$.  By the prime number theorem, we have
\begin{equation}\label{ineq:Matomaki-0}
\#([Y,2Y)\cap \mathcal{P}) \geq  \frac{Y}{2\log Y}
\end{equation}
since $Y\geq Y_0$ and $Y_0$ is sufficiently large. 
By decomposing $[Y,2Y)$ into at least $2Y^{1-\eta}$ intervals with length $Y^\eta/2$, there exist an interval $[X,X+X^\eta] \subseteq [Y,2Y)$ and a constant $d_3>0$ such that 
\[
\# ([X, X+X^\eta]\cap \mathcal{P}) \geq  \frac{d_3 X^\eta}{\log X}, 
\]
Since $X\geq Y \geq Y_0 \geq X_1$,  by applying Lemma~\ref{Lemma-Matomaki3} with $(e_k)_{k\ge 2} \coloneqq (c_k)_{k\ge 2}$, there exists a sequence $(p_k)_{k\ge 1}$ of prime numbers  such that for every $k\in \mathbb{N}$, we have 
\begin{equation}\label{ineq:Matomaki-3}
p_k^{c_{k+1}} \leq p_{k+1} < (p_k+1)^{c_{k+1}}-1.
\end{equation}
Therefore, since $C_k=c_1\cdots c_k$, for every $k\in \mathbb{N}$, by \eqref{ineq:Matomaki-3} we have
\begin{equation*}\label{ineq:Matomaki-4}
p_k^{1/C_k} \leq  p_{k+1}^{1/C_{k+1}} < (p_{k+1}+1)^{1/C_{k+1}}<(p_k+1)^{1/C_k}.
\end{equation*}
By substituting $k=1,2,\ldots$, we deduce that 
\begin{equation}\label{ineq:Matomaki-5}
p_1^{1/C_1} \leq  p_{2}^{1/C_{2}}\leq p_{3}^{1/C_{3}}\leq \cdots  < (p_{3}+1)^{1/C_{3}}<(p_{2}+1)^{1/C_{2}}<(p_1+1)^{1/C_1}.
\end{equation}
Thus, both $\lim_{k\to \infty} p_k^{1/C_k}$ and $\lim_{k\to \infty} (p_k+1)^{1/C_k}$ exist, say $A$ and $A'$, respectively. Then, we have $A\leq A'$ by \eqref{ineq:Matomaki-5}. Furthermore, by \eqref{ineq:Matomaki-5}, we conclude that 
\[
p_k\leq A_1^{C_k}\leq A_2^{C_k}<p_k+1,
\]
which implies that $p_k=\lfloor A^{C_k} \rfloor$ for every $k\in \mathbb{N}$.  Therefore, we obtain $A\in \mathcal{W}(C_k)$. Since $\mathcal{W}(C_k)$ is non-empty, the number $\xi(C_k)$ exists by Lemma~\ref{Lemma:existence}.
 \end{proof}


\section{Auxiliary results}\label{Section:Lemmas}

\subsection{A key dichotomy}
For all real numbers $x$, we define $\|x\|$ as the length between $x$ and the nearest integer of $x$.

\begin{lemma}\label{lemma:IdealIneq}
For all real numbers $x\geq 1$ and $c\geq 1$, if   
\[
\|x\|>  1/(c\lfloor x\rfloor^{c-1} ),
\]
then we have $\lfloor x \rfloor^{c} \leq \lfloor x^c \rfloor < (\lfloor x \rfloor+1)^c-1$.
\end{lemma}
\begin{proof}
We observe that 
\begin{align*}
\lfloor x \rfloor^{c} > \lfloor x^c \rfloor&\implies \lfloor x\rfloor^c > x^c -\{x^c\}\\
& \implies 1\geq  x^c- \lfloor x\rfloor^c.
\end{align*}
By the mean value theorem, there exists a real number $y\in [\lfloor x \rfloor,x]$ such that 
\[
x^c- \lfloor x\rfloor ^c= c(x-\lfloor x\rfloor) y^{c-1} \geq c\{x\} \lfloor x\rfloor^{c-1}\geq c\|x\| \lfloor x\rfloor^{c-1},
\]
and hence
\begin{equation}\label{ineq:IdealIneq-1}
\lfloor x \rfloor^{c} > \lfloor x^c \rfloor \implies \|x\| \leq 1/(c\lfloor x\rfloor^{c-1} ).
\end{equation}
In addition, we have 
\begin{align*}
\lfloor x^c \rfloor\geq (\lfloor x \rfloor+1)^c-1 &\implies 1\geq (\lfloor x \rfloor+1)^c-\lfloor x^c\rfloor\\
&\implies 1\geq (\lfloor x \rfloor+1)^c-x^c.
\end{align*}
 By applying the mean value theorem again, there exists a real number  $z\in [ x,\lfloor x\rfloor+1]$ such that 
\[
(\lfloor x \rfloor+1)^c-x^c=c(\lfloor x \rfloor+1-x)z^{c-1}\geq c(1-\{x\}) x^{c-1} \geq c\|x\| \lfloor x\rfloor^{c-1}
\]
and hence 
\begin{equation}\label{ineq:IdealIneq-2}
\lfloor x^c \rfloor \geq (\lfloor x \rfloor+1)^c-1 \implies \|x\| \leq 1/(c\lfloor x\rfloor^{c-1} ).
\end{equation}
Combining \eqref{ineq:IdealIneq-1} and \eqref{ineq:IdealIneq-2}, we conclude Lemma~\ref{lemma:IdealIneq}.
\end{proof}

Let $\theta$ be a real number in $[1/2,1)$ satisfying \eqref{dag}. Throughtout this subsection, let $(C_k)_{k\ge 1}$ be a sequence of real numbers satisfying \eqref{G1}, \eqref{G2}, and \eqref{G3}. Theorem~\ref{Theorem:Matomaki} implies that  $\mathcal{W}(C_k)$ is non-empty and $\xi(C_k)$ exists. Let $\xi=\xi(C_k)$, and let $p_k=\lfloor \xi^{C_k}\rfloor$ for every $k\in \mathbb{N}$. 

\begin{lemma}\label{lemma:key1} Suppose that for every  $k\in \mathbb{N}$ we have \eqref{ineq:key1-assump1}.
Let $\epsilon$ and $I$ be as in Proposition~\ref{Proposition:General}. Then, the following \eqref{CaseI} or \eqref{CaseII} is true\textup{:}
\begin{enumerate}  \renewcommand{\theenumi}{\Roman{enumi}}
\renewcommand{\labelenumi}{(\theenumi)}
\item \label{CaseI}
for infinitely many integers $k$, we have   
\begin{equation*}\label{ineq:caseA}
\|\xi^{C_k}\|\leq \frac{1}{c_{k+1}\lfloor \xi^{C_k}\rfloor^{c_{k+1}-1}},
\end{equation*} 
\item \label{CaseII} there exists $k_0>0$ such that for every $k\in I \cap[k_0,\infty)$, we have 
\begin{equation*}\label{ineq:caseB}
p_{k}^{c_{k+1}} \leq p_{k+1} \leq p_{k}^{c_{k+1}}+p_k^{\theta c_{k+1}}.  
\end{equation*}
\end{enumerate}
\end{lemma} 
 
\begin{proof}
Assume that both \eqref{CaseI} and \eqref{CaseII} are false. Since \eqref{CaseI} is false, there exists a positive integer $k_1$ such that for every $k\geq k_1$ we have
\begin{equation}\label{ineq:key1-f1}
\|\xi^{C_k}\|> \frac{1}{c_{k+1}\lfloor \xi^{C_k}\rfloor^{c_{k+1}-1}   }.
\end{equation}
Then, by Lemma~\ref{lemma:IdealIneq} with $x\coloneqq \xi^{C_k}$ and $c\coloneqq c_{k+1}$, for all $k\geq k_1$ we obtain 
\[
\lfloor \xi^{C_k} \rfloor ^{c_{k+1}} \leq  \lfloor \xi^{C_{k+1}}\rfloor <  (\lfloor \xi^{C_k} \rfloor +1)^{c_{k+1}}-1,
\]
which implies that for all $k\geq k_1$ 
\begin{equation}\label{ineq:key1-f2}
p_k^{c_{k+1}}\leq p_{k+1} <(p_k+1)^{c_{k+1}}-1.   
\end{equation}
Since \eqref{CaseII} is false,  there are infinitely many integers $k\in I$ such that 
\begin{equation}\label{ineq:key1-f3}
p_{k+1} \notin [p_{k}^{c_{k+1}} ,p_{k}^{c_{k+1}}+p_k^{\theta c_{k+1}}].
\end{equation}
Then, we set $\delta=  \min \{\xi^{C_k} -p_k \mid k=1,2,\ldots ,k_1\}$.  By \eqref{ineq:key1-assump1}, for every integer $k\leq k_1$, we have $p_k = \lfloor \xi^{C_k} \rfloor < \xi^{C_k}$, and hence $\delta>0$ holds. 

Let $\epsilon_0=1-\theta$, and let $X_1=X_1(\epsilon_0)$ be as in Lemma~\ref{Lemma-Matomaki3}. Take a sufficiently large integer $m\in I\cap[k_1,\infty)$  satisfying \eqref{ineq:key1-f3}, and
\begin{gather}\label{ineq:pm-X1} 
p_m^{c_{m+1}}\ge X_1, \\
\label{ineq:distance-xipk} \delta \geq  (2^{-1/C_m}\xi)^{C_{k_1}-C_m }. 
\end{gather}
We note that $I=\{k\in \mathbb{N}\colon c_{k+1} \geq 1/(1-\theta) +\epsilon \}$ is an infinite set by (\ref{G3}) and \eqref{Condition:epsilon}. 
Then, by \eqref{ineq:key1-f2} and \eqref{ineq:key1-f3}, we have 
\begin{equation}\label{ineq:keylem-0}
p_{m+1}>p_{m}^{c_{m+1}}+p_m^{\theta c_{m+1}}.
\end{equation}
 Let $q_m=p_m$. Since $\theta$ satisfies \eqref{dag} and $m$ is sufficiently large, there exists an absolute constant $d>0$ such that  
\[
\#([q_m^{c_{m+1}}, q_{m}^{c_{m+1}}+q_m^{\theta c_{m+1}}]\cap \mathcal{P}) \geq  \frac{dq_m^{\theta c_{m+1}}}{\log q_m^{c_{m+1}}}.
\]
Therefore, by \eqref{ineq:pm-X1} and Lemma~\ref{Lemma-Matomaki3} with $(e_2,e_3,\ldots)\coloneqq (c_{m+2},c_{m+3},\ldots)$, $\epsilon_0\coloneqq 1-\theta$, $X\coloneqq q_m^{c_{m+1}}$ , and $\eta\coloneqq \theta\in[1/2,1-\epsilon_0] $, there exists a sequence $(q_k)_{k\ge m}$ of prime numbers such that 
\begin{gather} \label{ineq:keylem-1}
  q_m^{c_{m+1}}\leq q_{m+1}\leq q_{m}^{c_{m+1}}+q_m^{\theta c_{m+1}}, \\ \label{ineq:keylem-2}
q_{k}^{c_{k+1}} \leq q_{k+1} < (q_{k}+1)^{c_{k+1}}-1
\end{gather}
for every integer $k\geq m+1$. We recall that $c_{m+1}\geq 1/(1-\theta)+\epsilon$ by the choice of $m$, and hence \eqref{ineq:keylem-1} implies that \eqref{ineq:keylem-2} also holds for $k= m$.  

Let $q_{k}=p_k$ for every $1\leq k\leq m-1 $. Then \eqref{ineq:key1-f2} implies that \eqref{ineq:keylem-2} is true for every integer $k\geq k_1$. Therefore, in a similar manner with \eqref{ineq:Matomaki-5}, we obtain 
\begin{equation}\label{ineq:keylem-3}
q_{k_1}^{1/C_{k_1}}\leq q_{k_1+1}^{1/C_{k_1+1}} \leq  \cdots  <(q_{k_1+1}+1)^{1/C_{k_1+1}}<(q_{k_1}+1)^{1/C_{k_1}},
\end{equation}
which implies that $\zeta= \lim_{k\to \infty} q_k^{1/C_k}$ exist. Furthermore, similary with the proof of Theorem~\ref{Theorem:Matomaki}, we obtain 
\begin{equation}\label{Equation:integerparts-zeta}
\lfloor \zeta^{C_k} \rfloor =q_k \in \mathcal{P}
\end{equation}
for every integer $k\ge k_1$. We note that 
\begin{equation}\label{ineq:zetaxi}
\zeta<\xi 
\end{equation}
holds since $\lfloor \zeta^{C_{m+1}}\rfloor=q_{m+1}<p_{m+1}=\lfloor \xi^{C_{m+1}} \rfloor$ by \eqref{ineq:keylem-0} and \eqref{ineq:keylem-1}.

Let us show that \eqref{Equation:integerparts-zeta} is true for every $k\in \mathbb{N}$. Since $\lfloor \zeta^{C_m} \rfloor = q_m=p_m =\lfloor \xi^{C_m} \rfloor$,  we have
\begin{equation}\label{ineq:zetaxi2}
1\geq |\xi^{C_m} - \zeta^{C_m}| = |\xi^{C_k C_m/C_k } - \zeta^{C_k C_m/C_k }|=(\xi^{C_k})^{C_m/C_k } - (\zeta^{C_k})^{ C_m/C_k } 
\end{equation}
for every $1\leq k\leq k_1$, where the last equation follows from \eqref{ineq:zetaxi}. The mean value theorem yields that for every $1\leq k\leq k_1\leq m$
\[
(\xi^{C_k})^{C_m/C_k } - (\zeta^{C_k})^{ C_m/C_k } \geq  (\xi^{C_k}-\zeta^{C_k})\zeta^{C_m-C_k},
\]  
and hence
\begin{equation}\label{ineq:keylem-5}
 \xi^{C_k}-\zeta^{C_k} = |\xi^{C_k} - \zeta^{C_k}|\leq \zeta^{C_k-C_m}\leq \zeta^{C_{k_1}-C_m}.
\end{equation}
 Since \eqref{ineq:zetaxi2} leads to $\zeta^{C_m} \ge \xi^{C_m} -1 \ge   \xi^{C_m}/2$, we have  $\zeta \ge  2^{-1/C_m}\xi $. 
Therefore, by \eqref{ineq:keylem-5}, we obtain 
\begin{equation}\label{ineq:keylem-6}
 \xi^{C_k}-  (2^{-1/C_m}\xi)^{C_{k_1}-C_m}\leq \xi^{C_k}- \zeta^{C_{k_1}-C_m}\leq \zeta^{C_k} .
\end{equation}
By the definition of $\delta$, for every $1\leq k\leq k_1$, we have
\begin{align*}
q_k&=p_k\leq  \xi^{C_k} -\delta_k \overset{\eqref{ineq:distance-xipk}}{\leq} \xi^{C_k}-(2^{-1/C_m}\xi)^{C_{k_1}-C_m}\\
&\overset{  \eqref{ineq:keylem-6}}{\leq} \zeta^{C_k}<\xi^{C_k}<\lfloor \xi^{C_k} \rfloor +1=p_k+1=q_k+1,
\end{align*}
and hence \eqref{Equation:integerparts-zeta} is true for all $k\in \mathbb{N}$. We conclude that $\zeta\in \mathcal{W}(C_k) $, a contradiction to the minimality of $\xi$.
\end{proof}

\begin{lemma}\label{lemma:xi-fractional}
Suppose that for every  $k\in \mathbb{N}$ we have \eqref{ineq:key1-assump1}.
Let $\epsilon$ and $I$ be as in Proposition~\ref{Proposition:General}. If \eqref{CaseI} is false, then for every sufficiently large  $k \in I$, we have
\begin{equation}\label{ineq:xi-result1}
\{ \xi^{C_k} \} \leq \frac{2}{c_{k+1}\lfloor \xi^{C_k}\rfloor^{(1-\theta)c_{k+1}-1}  }. 
\end{equation}
\end{lemma}
\begin{proof}
Assume that \eqref{CaseI} is false. By Lemma~\ref{lemma:key1}, there exists $k_0>0$ such that for every $k\in I \cap[k_0,\infty)$, we obtain \eqref{ineq:caseB}. Then, for every sufficiently large $k\in I$, we have 
\begin{equation}\label{Inequality:xi-frac:1}
p_k=\lfloor \xi^{C_k} \rfloor \leq  \xi^{C_{k}},     
\end{equation}
Furthermore, by \eqref{ineq:caseB}, for every sufficiently large $k\in I$, we have
\[
\xi^{C_{k+1}}\leq p_{k+1}+1 \leq p_{k}^{c_{k+1}}+p_k^{\theta c_{k+1}}+1,
\]
which implies that 
\begin{equation}\label{Inequality:xi-frac:2}
\begin{aligned}
\xi^{C_k} 
&\leq  (p_{k}^{c_{k+1}}+p_k^{\theta c_{k+1}}+1)^{1/c_{k+1}} \\
&\leq p_k (1+2p_k^{(\theta-1)c_{k+1}} )^{1/c_{k+1}} \\
&\leq p_k + \frac{2}{c_{k+1}p_k^{(1-\theta)c_{k+1}-1 }}. 
\end{aligned}
\end{equation}
Therefore, by \eqref{Inequality:xi-frac:1} and \eqref{Inequality:xi-frac:2}, we conclude \eqref{ineq:xi-result1}.
\end{proof}

\subsection{Results from Diophantine approximation}

\begin{theorem}\label{Theorem-Dubickas1}Let $\alpha >1$ be an algebraic number, and $q$ be a positive integer. Suppose that $0<s_1<s_2<\cdots$ is a sequence of positive integers. Then, either for some $m\geq 0$ the number $\alpha^{s_m}$ is a Pisot number or for each $\epsilon>0$  there exists a positive integer $k_0 = k_0(\epsilon)$ such that 
\[
\| q\alpha^{s_k}  \|> e^{-\epsilon s_k }
\]
for every $k\geq k_0$.
\end{theorem}

\begin{proof}
See \cite[Lemma~6]{Dubickas2022} which is given by Dubickas. This theorem essentially follows from the result \cite{CorvajaZannier} by Corvaja and Zannier.
\end{proof}

By Lemmas~\ref{lemma:key1} and \ref{lemma:xi-fractional}, both cases \eqref{CaseI} and \eqref{CaseII} imply that  
\[
\limsup_{k\to \infty} \|\xi^{C_k}\|^{1/C_k} <1.
\]
By combining this inequality and Theorem~\ref{Theorem-Dubickas1} with $q=1$ and $s_k=C_k$, either $\xi$ is transcendental or there exists $m\in \mathbb{N}$ such that $\xi^{C_m}$ is a Pisot number. In the latter case, we apply Baker's theorem. We say that the \textit{height} of an algebraic number is the maximum of the absolute values of the relatively prime integer coefficients in its minimal polynomial.

\begin{theorem}[Baker's theorem {\cite[Theorem 3.1]{Baker}}]\label{Theorem:Baker}
Let $\alpha_1,\ldots,\alpha_n$ be non-zero algebraic numbers with degrees at most $d$ and heights at most $A$. Further, let $\beta_0, \beta_1,\ldots, \beta_n$ be algebraic numbers with degrees at most $d$ and heights at most $B(\ge 2)$. Let 
\[
\Lambda = \beta_0 +\beta_1 \log \alpha_1 + \cdots + \beta_n\log \alpha_n.
\]
Then either $\Lambda=0$ or $|\Lambda|>B^{-C}$, where $C$ is an effectively computable number depending only on $n$, $d$, $A$ and the original determinations of the logarithms.
\end{theorem}

\subsection{Lemmas for Pisot numbers}

\begin{lemma}\label{Lemma-GoldenRatio}
Let $\beta_1$ be a Pisot number of degree $2$, and let $\beta_2$ be the conjugate of $\beta_1$ over $\mathbb{Q}$. If $\beta_1+\beta_2=1$, then we have $\beta_1=\frac{1+\sqrt{5}}{2}$ and $\beta_2=\frac{1-\sqrt{5}}{2}$. 
\end{lemma}

\begin{proof}
Let $N=-\beta_1\beta_2\in \mathbb{Z}\setminus\{0\}$. Then, we have $(x-\beta_1)(x-\beta_2)= x^2-x -N$. Therefore, by solving $x^2-x -N=0$, we obtain
\[
\beta_1= \frac{1+\sqrt{1+4N}}{2}, \quad \beta_2= \frac{1-\sqrt{1+4N}}{2}.
\]
By $|\beta_2|<1$, we conclude that $N=1$. 
\end{proof}

\begin{lemma}[{\cite[Lemma 3]{Dubickas2002}}]\label{Lemma:PisotPowers}
Let $\ell\in \mathbb{N}$, and let $\beta$ be a Pisot number of degree $\ell$. For all $m\in \mathbb{N}$, $\beta^m$ is a Pisot number of degree $\ell$. 
\end{lemma}

\begin{proof}
Let $\beta_1=\beta,\beta_2,\ldots, \beta_\ell$ be all the conjugates of $\beta$. Let 
\begin{equation}\label{Equation:PisotPowers1}
Q(X)=\prod_{j=1}^\ell(X-\beta_j^m)\in \mathbb{Z}[X], 
\end{equation}
and let $P\in \mathbb{Z}[X]$ be the minimum polynomial of $\beta^m$. Since $\beta^m$ is an algebraic integer,  we may assume that $P$ is a monic polynomial. It is clear that $\deg (\beta^m)\leq \ell$, and so we now suppose that $\deg (\beta^m)<\ell$. By this supposition, there exists a non-constant polynomial $R\in \mathbb{Z}[X]$  such that $Q(X)=P(X)R(X)$. By \eqref{Equation:PisotPowers1},  $R(X)=\prod_{j\in S} (X-\beta_j^m)$ for some non-empty set $S\subset \{2,3,\ldots, \ell\}$.  We note that $R(0)$ is an integer, but 
\[
0<\prod_{j\in S}|\beta_j^m|  <1 
\] 
since $\beta$ is a Pisot number, a contradiction. 
\end{proof}

\begin{lemma}\label{Lemma:Trace} Let $\beta$ be a Pisot number of degree $\ell$, and let $\beta_1=\beta, \beta_2,\ldots, \beta_\ell$ be the conjugates of $\beta$ over $\mathbb{Q}$. Then, there exists $s_0=s_0(\beta)>0$ such that for all integers $s\geq s_0$, if $\{\beta^{s} \}<1/2$, then  $\mathrm{Tr}(\beta^s)=\beta_{1}^{s}  +\cdots + \beta_{\ell}^{s }= \lfloor \beta^s \rfloor$.
\end{lemma}

\begin{proof}It is clear that $\mathrm{Tr}(\beta^s)=\beta_{1}^{s}  +\cdots + \beta_{\ell}^{s}$ from Lemma~\ref{Lemma:PisotPowers}. Since $\beta$ is a Pisot number, we have $|\beta_j|<1$ for all $j=2,3,\ldots, \ell$. Therefore, there exists $s_0=s_0(\beta)>0$  such that for all integers $s\geq s_0$ we have $|\beta_2^s+\cdots + \beta_\ell^s|<1/4$. We now take an arbitrary integer $s$ with $s\geq s_0$. Suppose that $\{\beta^s\}<1/2$. 

Since $\beta_1^s+\cdots + \beta_\ell^s$ is a non-zero integer from the theory of fields and  
\[
\beta^s=\beta_1^s= \beta_1^s+\cdots + \beta_\ell^s- (\beta_2^s+\cdots + \beta_\ell^s),
\]
we see that 
\[
\{\beta^s\}= \begin{cases}
-(\beta_2^s+\cdots + \beta_\ell^s) & \text{if $\beta_2^s+\cdots + \beta_\ell^s\leq 0$},\\
1- (\beta_2^s+\cdots + \beta_\ell^s) & \text{otherwise}.
\end{cases}
\]
The latter case naver happens since $\{\beta^s\}<1/2$ and $|\beta_2^s+\cdots + \beta_\ell^s|<1/4$, and hence 
$\{\beta^s\}=-(\beta_2^s+\cdots + \beta_\ell^s)$. Therefore, we conclude that $ \lfloor \beta^s \rfloor=\beta^s-\{\beta^s\} =\beta_{1}^{s}  +\cdots + \beta_{\ell}^{s }$.
\end{proof}

\begin{lemma}[{\cite[Lemma~8]{BugeaudDubickas}}]\label{lemma-BD1}
Let $\alpha$ be a real number greater than $1$. Let $n$ and $m$ be positive integers.  If $\alpha^n$ and $\alpha^m$ are Pisot numbers, then $\alpha^{\gcd(n,m)}$ is a Pisot number.  
\end{lemma}
\begin{proof}
We may assume that $\gcd(n,m)=1$ by replacing $n/\gcd(n,m)$ and $m/\gcd(n,m)$ with $n$ and $m$, respectively.  Since $\alpha^n$ is an algebraic integer, $\alpha$ is so. 

Suppose that $\alpha$ is not a Pisot number. Then, we may take a conjugate $\beta$ of $\alpha$ satisfying $|\beta|>1$ and $\beta\neq \alpha$. Thus, we have $\beta^n=\alpha^n$ since  $\beta^n$ is also a conjugate of $\alpha^n$ and $\alpha^n$ is a Pisot number. Therefore, $\beta=\alpha e^{2\pi i k/n}$ for some $k\in \{1,2,\ldots,n-1\}$. Similarly, we have $\beta^m=\alpha^m$, and hence 
\[
\alpha^m e^{2\pi i k m/n} =\beta^m=\alpha^m. 
\]
This implies  $\exp(2\pi i k m/n)=1$, a contradiction since $\gcd(n,m)=1$. 
\end{proof}

\begin{lemma}\label{Lemma:Mignotte}
Let $\beta$ be a Pisot number of degree $\ell\ge 2$ with conjugates $\beta_1=\beta, \beta_2,\ldots ,\beta_\ell$. For every $k$ with $ \beta_k \notin \mathbb{R}$, if $\beta_k=|\beta_k|e^{2\pi i \phi}$ for some $\phi\in \mathbb{R}$, then $\phi$ is irrataional. 
\end{lemma}

\begin{proof}
Fix any $k$ with $\beta_k\notin \mathbb{R}$, and let $\beta_k=|\beta_k|e^{2\pi i \phi}$ for some $\phi\in \mathbb{R}$. Then, $\phi\notin \mathbb{Z}$. Then, there exists $k'\in [1,\ell]$ with $k\neq k'$ such that $\beta_{k'}=\overline{\beta_k}=|\beta_k| e^{-2\pi i \phi}$. Suppose that $\phi$ is rational and let $\phi=a/b$ for some $a\in \mathbb{Z}$ and $b\in \mathbb{N}$. Then, by Lemma~\ref{Lemma:PisotPowers},  $\beta_1^b,\ldots, \beta_\ell^b$ are pairwise distinct. However, we have
\[
\beta_k^b= |\beta_k|^b e^{2\pi i a}=|\beta_k|^b =|\beta_k|^b e^{-2\pi i a} = \beta_{k'}^b,
\]
a contradiction.
\end{proof}

\begin{remark}
This lemma follows from Mignotte's result \cite{Mignotte}, which states that for $N_1,\ldots, N_\ell \in \mathbb{Z}$ if $\beta_1^{N_1}\cdots \beta_\ell^{N_\ell}=1$ holds,  then $N_1=\cdots =N_\ell=0$. 
\end{remark}

\begin{lemma}[{\cite{Smyth}}]\label{Lemma:Smyth}
Let $\beta$ be a Pisot number of degree $\ell \geq 2$ with conjugates $\beta_1=\beta, \beta_2,\ldots ,\beta_\ell$. For $i\neq j$, if $|\beta_i|=|\beta_j|$, then $\beta_i=\overline{\beta_j}$. \end{lemma}

\begin{proof}Let 
\begin{equation}\label{Equation:Smyth:1}
f(X)= \prod_{1\leq i\leq j \leq \ell} (X-\beta_i\beta_j). 
\end{equation}
Then, it is clear that $f(X)\in \mathbb{Z}[X]$ since $f(X)$ is symmetric and $\beta$ is an algebraic integer. Assume that $|\beta_i|=|\beta_j|$ but $\beta_i \neq \overline{\beta_j}$ for some $i\neq j$. Then, both $(X-\beta_i\overline{\beta_i})$ and $(X-\overline{\beta_j}\beta_j)$ appear in the product \eqref{Equation:Smyth:1}. Therefore, $(X-\phi)^2$ devides $f(X)$, where $\phi\coloneqq \beta_i\overline{\beta_i}=\overline{\beta_j}\beta_j$. Since $\phi$ is an algebraic integer, the product of all the conjugates of $\phi$ is a non-zero integer. In addition, $|\phi|=|\beta_i|^2<1$ since $\beta$ is a Pisot number. Therefore, we find one of the conjugates $\phi'$ of $\phi$ such that $|\phi'|>1$. Then, $(X-\phi')^2$ devides $f(X)$. However, a root of $f(X)$ with absolute value strictly greater than one forms as $\beta \beta_j$ for some $1\leq j\leq \ell$. All of such roots are distinct, a contradiction.  
\end{proof}

\begin{lemma}\label{Lemma:non-Liouville}
Let $\alpha$ be a non-real algebraic number. Suppose that there exists an irrational number $\psi\in (0,1)$ such that $\alpha=|\alpha|e^{2\pi i \psi}$. Then, there exists $\lambda=\lambda(\psi)>0$ such that, for all $N\ge 2$ and for all integers $p$ and $q$ with $\max(p,q)\leq N$, we have  
\begin{equation}\label{ineq:non-Liouville}
|q\psi-p| \geq N^{-\lambda}.
\end{equation}
\end{lemma}

\begin{proof}We now define $\log z=\log r + 2\pi i x$ if $z=re^{2\pi i x}$ for some $r\ge 0$ and $x\in [0,1)$. We oberve that $\log (\alpha/|\alpha|)$ and $\log (-1)$ are linearly independent over $\mathbb{Q}$ since $2\pi i \psi= \log (\alpha/|\alpha|)$, $\pi i =\log(-1)$, and $\psi$ is irrational. Therefore, Theorem~\ref{Theorem:Baker} with $\Lambda\coloneqq q \log(\alpha/|\alpha|)  -p\log (-1)$ implies that there is a constant $K>0$ such that for all $N\ge 2$ and for all integers $p$ and $q$ with $\max(p,q)\leq N$, we obtain 
\[
2\pi|q\psi -p  | =|q \log(\alpha/|\alpha|)  -p\log (-1)  |\ge  N^{-K}.
\]
This completes the proof of Lemma~\ref{Lemma:non-Liouville}. 
\end{proof}

\begin{lemma}[{\cite[Lemma~8]{Dubickas2022}}]\label{lemma-Dubickas2}
Let $\beta$ be a Pisot number of degree $\ell \geq 2$ with conjugates $\beta_1=\beta, \beta_2,\ldots ,\beta_\ell$ labelled so that $\beta_1>1>|\beta_2|\geq  \cdots \geq |\beta_\ell|$. Then, there is a real number $\lambda>0$ depending only on $\beta$ such that
\[
|\beta_{2}^n +\cdots +\beta_{\ell}^n|\geq |\beta_2|^n n^{-\lambda} 
\]
for each sufficiently large integer $n$.
\end{lemma}

\begin{proof}
We may assume that $\ell \ge 3$. 
If $|\beta_2|>|\beta_3|$, then the lemma is trivial since 
\[
|\beta_{2}^n +\cdots +\beta_{\ell}^n|\geq |\beta_{2}|^n - \ell|\beta_{3}|^n \ge |\beta_{2}|^n/2 
\] 
for every sufficiently large $n$. 
If $|\beta_2|=|\beta_3|$, then Lemma~\ref{Lemma:Smyth} yields that $\beta_3=\overline{\beta_2}$ and $|\beta_2|>|\beta_4|$. Especially, $\beta_2$ is not real. Then, we obtain
\[
|\beta_{2}^n +\cdots +\beta_{\ell}^n|\geq |\beta_2|^n \biggl(|1 + (\beta_3/\beta_2)^n| - \ell (|\beta_4|/|\beta_2|)^n \biggr).
\]
Let $\beta_2=|\beta_2|e^{2\pi i \phi}$ for some $\phi\in (0,1)$. Further, Lemma~\ref{Lemma:Mignotte} implies that $\phi$ is irrational. For every $n\in \mathbb{N}$, we have 
\[
|1 + (\beta_3/\beta_2)^n|=|1+e^{4\pi i n \phi}| = 2 |\cos(2\pi n \phi )  |=2 |\sin(\pi/2 -2\pi n \phi ) |\ge 4  \left\|1/2 - 2 n \phi     \right \|.
\]
Lemma~\ref{Lemma:non-Liouville} implies that there are constants $K_1,K_2>0$ such that for every $n\ge 2$, we obtain 
\[
\left\|1/2 - 2 n \phi     \right \|\ge K_1 n^{-K_2}.
\]
Therefore, there exists $\lambda>0$ such that for every sufficiently large $n$, we have 
\[
|\beta_{2}^n +\cdots +\beta_{\ell}^n|
\ge |\beta_2|^n n^{-\lambda}. \qedhere
\]
\end{proof}

In \cite{Dubickas2022}, Dubickas proved the transcendency of certain constants deduced from some recurrent sequences by applying \cite[Lemma~6]{Dubickas2022} and \cite[Lemma~8]{Dubickas2022} (Theorem~\ref{Theorem-Dubickas1} and Lemma~\ref{lemma-Dubickas2}). This is an improvement on the result regarding irrationality by Wagner and Ziegler \cite{WagnerZiegler}.  A basic idea for the proof is derived from these works and \cite{BugeaudDubickas} by Bugeaud and Dubickas. They studied the set of limit points of $\|\alpha^n\|^{1/n}$ $(n=1,2,3,\ldots)$. 

\section{Proof of Proposition~\ref{Proposition:General}}\label{Section:ProofGenProp}

\begin{lemma}\label{Lemma:PisotNumber}
Let $\alpha>1$ be an algebraic number, and let $0<s_1<s_2<\cdots$ be a sequence of positive real numbers satisfying that $s_k\in \mathbb{N}$ for sufficiently large $k$. Let $(t_k)_{k\ge 1}$ be a sequence of positive real numbers with $\liminf_{k\to\infty} t_k>0$. Suppose that there exists a constant $K>0$ such that for every sufficiently large $k\in \mathbb{N}$ 
\begin{equation}\label{ineq:PisotNumber-1}
\|\alpha^{s_k}\| \leq K \alpha^{-s_k t_{k}}.  
\end{equation}
Then,  there exists $g\in \mathbb{N}$ such that $\alpha^g$ is a Pisot number of degree $\ell$ with
 \[
1\leq \ell \leq 1+\biggl(\limsup_{k\to \infty} t_k\biggr)^{-1}
 \]
 and for every sufficiently large $k$, we have $g\mid s_k$.  Furthermore, if we have 
\[
\ell=1+\biggl(\limsup_{k\to \infty} t_k\biggr)^{-1}, 
\]
then $\alpha^g$ is a unit Pisot number.
\end{lemma}

\begin{proof}
Let $\overline{t}=\limsup_{k\to\infty}\: t_k$ and $\underline{t}=\liminf_{k\to\infty}\: t_k$.  By \eqref{ineq:PisotNumber-1} and Theorem~\ref{Theorem-Dubickas1} with $\epsilon\coloneqq  (\underline{t}/2)\log \alpha$ and $q\coloneqq 1$, there exists a positive integer $m$ such that $\alpha^{s_m}$ is a Pisot number. Let $g$ be the smallest positive integer such that  $\alpha^g$ is a Pisot number. 

Let us show that $g\mid s_k$ for every sufficiently large $k$. If not, there are infinitely many positive integers $n_1<n_2<n_3<\cdots$ such that $g \nmid s_{n_k}$. Then, by \eqref{ineq:PisotNumber-1} and Theorem~\ref{Theorem-Dubickas1} with $s_k\coloneqq s_{n_k}$, $\epsilon\coloneqq  (\underline{t}/2)\log \alpha$ and $q\coloneqq 1$, there exists $r\in \mathbb{N}$ such that 
$\alpha^{s_{n_r}}$ is a Pisot number. By Lemma~\ref{lemma-BD1} and the choice of $g$, we obtain $g\mid s_{n_r}$, a contradiction.

Let $\beta=\alpha^g$, and let $\beta_1, \beta_2,  \ldots ,  \beta_\ell$ be the conjugates over $\mathbb{Q}$ of $\beta$ labeled so that $\beta=\beta_1>1>|\beta_2|\geq  \cdots \geq |\beta_\ell|$. Therefore, \eqref{ineq:PisotNumber-1} and Lemma~\ref{lemma-Dubickas2} with $n\coloneqq s_{k}/g\in \mathbb{N}$ imply that for every sufficiently large $k\in \mathbb{N}$ we have
\begin{align*}
 K \alpha^{-s_{k} t_{k}}&\geq \|\alpha^{s_{k}}\| = \|\beta^{s_{k}/g} \|\\
&=\|\beta_2^{s_{k}/g}+\cdots + \beta_\ell^{s_{k}/g}  \|
\geq |\beta_2|^{s_{k}/g}(s_{k}/g)^{-\lambda}
\end{align*}
for some constant $\lambda=\lambda(\beta)>0$. Thus, by taking the logarithm on both sides and taking $\limsup_{k\to \infty}$, we obtain 
\[
\beta_1^{-\overline{t}/g}= \alpha^{-\overline{t}}\geq |\beta_2|^{1/g},
\]
which yields that $\beta_1^{-\overline{t}}\geq |\beta_2|$.
Since $\beta_1\cdots \beta_\ell\in \mathbb{Z}\setminus\{0\}$, we obtain
\[
1\leq | \beta_1\cdots \beta_\ell|\leq \beta_1^{-(\ell-1)\overline{t}+1 },
\]  
and hence $1\leq \ell \leq 1+1/\overline{t}$. Furthermore, if $\ell=1+1/\overline{t}$ is true, then  $1\leq  | \beta_1\cdots \beta_\ell|\leq \beta_1^0=1$, which means that $\beta$ is a Pisot unit. 
\end{proof}

\begin{lemma}\label{lemma:Case-A}
Let $(c_k)_{k\ge 1}$ be a sequence of real numbers satisfying \textup{(\ref{G1})} to \textup{(\ref{G4})}. We further suppose that for every  $k\in \mathbb{N}$ we have \eqref{ineq:key1-assump1}. If \eqref{CaseI} in Lemma~\ref{lemma:key1} is true, then $\xi(C_k)$ is transcendental.
\end{lemma}

\begin{proof}
Let $\xi=\xi(C_k)$, and we assume that $\xi$ is algebraic. We furthermore assume that \eqref{CaseI} is true, and so
\begin{equation}\label{eq:CaseA-1}
J\coloneqq \left\{n\in \mathbb{N} \colon \|\xi^{C_n}\|\leq  \frac{1}{c_{n+1}\lfloor \xi^{C_n}\rfloor^{c_{n+1}-1}} \right\}
\end{equation}
is infinite. Let $J=\{n_1,n_2,\ldots\}$, where $n_1<n_2<\cdots$. By Lemma~\ref{Lemma:PisotNumber} with $\alpha\coloneqq \xi$, $s_k\coloneqq C_{n_k}$, and $t_k\coloneqq c_{n_k+1}-1$, there exists $g\in \mathbb{N}$ such that $\xi^g=\beta$ is a Pisot number of degree $\ell$ with 
\begin{equation}\label{eq:CaseA-2}
1\leq \ell \leq 1+\biggl(\limsup_{\substack{k\to \infty} } c_{n_k+1}-1 \biggr)^{-1}
\end{equation}
and $g\:|\: C_{n}$ for every sufficiently large $n\in J$. By \eqref{G2} and \eqref{eq:CaseA-2}, we obtain $\ell\in\{1,2\}$. Suppose that $\ell=1$, which means that $\beta$ is an integer. In this case, for sufficiently large $n\in J$, we see that $\beta^{C_{n} /g} $ is a composite number, but
\[
\beta^{C_{n} /g} =  \xi^{C_{n}}= \lfloor \xi^{C_{n}}\rfloor 
\]
is a prime number, a contradiction. 

Suppose that $\ell=2$. Then, by \eqref{G2} and \eqref{eq:CaseA-2}, we have
\[
 \ell = 1+\biggl(\limsup_{\substack{k\to \infty} } c_{n_k+1}-1 \biggr)^{-1}.
\]
By Lemma~\ref{Lemma:PisotNumber} and $\ell=2$, $\beta$ is a quadratic Pisot unit, and hence $\|\beta^n\|=\beta^{-n}$ for all $n\in \mathbb{N}$. Therefore, by \eqref{G2},   for every sufficiently large $n\in J$, we obtain 
\[
\xi^{-C_{n}}=\beta^{-C_{n}/g} =\|\beta^{C_{n}/g} \| =  \|\xi^{C_n}\| \leq \frac{1}{c_{n+1} \lfloor \xi^{C_{n}}  \rfloor^{c_{n+1}-1} }\leq \frac{1}{2 \lfloor \xi^{C_n} \rfloor},
\]
a contradiction.
\end{proof}

\begin{proof}[Proof of Proposition~\ref{Proposition:General}]Let $\theta$ be a real number in $[1/2,1)$ satisfying \eqref{dag}. Let $(c_k)_{k\ge 1}$ be a sequence of real numbers satisfying \eqref{G1} to \eqref{G4}. By \eqref{G1} and \eqref{G2}, Theorem~\ref{Theorem:Matomaki} yields that $\mathcal{W}(C_k)$ is non-empty and $\xi=\xi(C_k)$ exists.

We suppose that for every  $m\in \mathbb{N}$, we have \eqref{ineq:key1-assump1}, and we further suppose that $\xi$ is algebraic. Let $\epsilon$ be a real number with \eqref{Condition:epsilon}, and let $I=\{k\in \mathbb{N} \colon c_{k+1}\geq 1/(1-\theta)+\epsilon \}$. \\[-5pt] 

\noindent\textbf{Proof of \eqref{Result:General:1}:}   By Lemma~\ref{lemma:key1}, we obtain \eqref{CaseI} or \eqref{CaseII}. Furthermore, Lemma~\ref{lemma:Case-A} implies that \eqref{CaseI} is false since $\xi$ is algebraic. Therefore, by Lemma~\ref{lemma:xi-fractional}, we have \eqref{ineq:xi-result1} for sufficiently large $k\in I$, and hence we obtain \eqref{Result:General:1}.  \\[-5pt]

\noindent\textbf{Proof of \eqref{Result:General:2}:} Let $\overline{c}=\limsup_{k\to\infty}c_{k+1}$. Let $I=\{n_1,n_2,\ldots\}$, where $0<n_1<n_2<\cdots$. By the definition of $I$, we have
\[
\limsup_{k\to \infty} c_{n_k+1}=\overline{c}\geq 1/(1-\theta)+\epsilon.
\] 
By $\|x \|\leq \{x\}$ and \eqref{Result:General:1}, for sufficiently large $k\in \mathbb{N}$, we have 
\begin{equation}
\| \xi^{C_{n_k}} \|\leq \{\xi^{C_{n_k}}\} \leq \frac{2}{c_{{n_k}+1}\lfloor \xi^{C_{n_k}}\rfloor^{(1-\theta)c_{{n_k}+1}-1}  }, 
\end{equation}
and hence by Lemma~\ref{Lemma:PisotNumber} with $\alpha\coloneqq \xi$, $s_k\coloneqq C_{n_k}$, and $t_k\coloneqq (1-\theta)c_{n_k+1}-1$, there exists $g\in \mathbb{N}$ such that $\xi^g$ is a Pisot number of degree $\ell$ with
\begin{equation}\label{inequality:ProofGen:1}
1\leq \ell\leq  1+\biggl((1-\theta)\overline{c}-1\biggl)^{-1},
 \end{equation}
and for sufficiently large $k\in \mathbb{N}$, we have $g\mid C_{n_k}$. 

Let $\beta=\xi^g$. If $\ell=1$, then $\beta$ is an integer. Thus, for sufficiently large $k\in I$ we have $\xi^{C_{k}} = \beta^{C_{k}/g}$ is an integer, a contradiction to \eqref{ineq:key1-assump1}. Thus, $\ell\neq 1$. Therefore, we obtain \eqref{Result:General:2}. \\[-5pt] 

\noindent\textbf{Proof of \eqref{Result:General:4}:} By \eqref{Result:General:1} and \eqref{Result:General:2}, for sufficiently large $k\in I$, we obtain $g \mid C_k$ and  
\[
\{\beta^{C_{k}/g} \}=\{\xi^{C_{k}} \}<1/2,
\]
and hence Lemma~\ref{Lemma:Trace} implies \eqref{Result:General:4}. \\[-5pt]

\noindent\textbf{Proof of \eqref{Result:General:3}:}  Suppose that $\ell=2$. Since $\beta_1 \beta_2\in \mathbb{Z}\setminus \{0\}$ and $\beta_1\in \mathbb{R}$, we have $\beta_2\in \mathbb{R} \setminus\{0\}$. By \eqref{Result:General:2} and \eqref{Result:General:4}, for every sufficiently large $k\in I$, we  obtain
\[
 \beta_2^{C_{k}/g}= \lfloor \beta_1^{C_{k}/g} \rfloor-\beta_1^{C_{k}/g} =-\{\beta_1^{C_{k}/g} \} <0,
\]
which yields that $\beta_2<0$, and $C_{k}/g$ is odd. Choose a sufficiently large $k\in \mathbb{N}$. Let $r$ and $s$ be positive odd numbers with $C_{k}/g=rs$.  Then, we have 
\begin{align*}
\lfloor\xi^{C_{k}}\rfloor&=\lfloor\beta_1^{C_{k}/g}\rfloor = \beta_1^{C_{k}/g}+\beta_2^{C_{k}/g}=\beta_1^{rs}+\beta_2^{rs}\\
&= (\beta_1^r+\beta_2^r)(\beta_1^{r(s-1)} -\beta_1^{r(s-2)}\beta_2^r+ \cdots -  \beta_1^r\beta_2^{r(s-2)}+\beta_2^{r(s-1)}).  
\end{align*}
By choosing $r=1$ and $s=C_k/g$, we have $\beta_1+\beta_2=1$ since $\lfloor\xi^{C_{k}}\rfloor$ is a prime number and
\[
\beta_1^{s-1} -\beta_1^{s-2}\beta_2+ \cdots -  \beta_1\beta_2^{s-2}+\beta_2^{s-1}
\]
is an integer greater than or equal to $\beta_1^{C_k/g-1}(\ge 2)$. 
Therefore, Lemma~\ref{Lemma-GoldenRatio} implies that $\beta_1=(1+\sqrt{5})/2$. In addition, if $C_{k}/g$ is a composite number, then we may choose $r$ and $s$ as odd numbers $\ge 3$. By $\beta_2=-0.618\cdots$, we observe that 
\begin{gather*}
\beta_1^r+\beta_2^r\ge \beta_1^3+\beta_2^3=4,\\
\beta_1^{r(s-1)} -\beta_1^{r(s-2)}\beta_2^r+ \cdots -  \beta_1^r\beta_2^{r(s-2)}+\beta_2^{r(s-1)}\ge \beta_1^{r(s-1)}\ge \beta_1^3\ge 4,
\end{gather*}
and hence $\lfloor\xi^{C_{k}}\rfloor$ is a composite number, a contradiction. Therefore, we conclude \eqref{Result:General:3}. 
\end{proof}

\section{Type A: Proof of Theorem~\ref{Theorem:Main1}}\label{Section:TypeA}

Suppose that a sequence $(c_k)_{k\ge 1}$ satisfies \eqref{A1} to \eqref{A4} in Theorem~\ref{Theorem:Main1}. Theorem~\ref{Theorem:BHP} yields that \eqref{dag} is true for $\theta=21/40$. Therefore, $(c_k)_{k\ge 1}$ satisfies \eqref{G1} to \eqref{G4} in Proposition~\ref{Proposition:General}, and so the constant $\xi=\xi(C_k)$ exists. 

If there exists $m\in \mathbb{N}$ such that $\xi^{C_m}\in \mathbb{N}$, then we have \eqref{equation:TypeA:1} for $\ell=1$ since a positive integer is a Pisot number of degree $1$. Thus, $\xi$ is an algebraic integer. In particular, if $\xi$ is rational, then $\xi$ must be a rational integer. This is a contradiction to   
\[
\xi^{C_k}=\lfloor \xi^{C_k}\rfloor \in \mathcal{P}.
\]
for every $k\in \mathbb{N}$. Therefore, $\xi$ is irrational.

We may assume that $\xi^{C_m} \notin \mathbb{N}$ for all $m\in \mathbb{N}$. Suppose that $\xi$ is algebraic. Then, by applying \eqref{Result:General:2} in Proposition~\ref{Proposition:General}, there exists $g\in \mathbb{N}$ such that $\beta=\xi^g$ is a Pisot number of degree $\ell\ge 2$ with \eqref{equation:TypeA:1}, and for every sufficiently large $k\in I$, we have $g\mid C_k$. Therefore, by applying  Lemma~\ref{Lemma:PisotPowers}, there exists $m\in \mathbb{N}$ such that $g\mid C_m$ and $\xi^{C_m}=\xi^{g\cdot C_m/g }$ is a Pisot number of degree $\ell$. In particular, $\xi$ is irrational since the degree of $\xi^{C_m}$ is greater than or equal to $2$.

\section{Type B: Proof of Theorem~\ref{Theorem:Main2} and Corollary~\ref{Corollary:Recurrence}}\label{Section:TypeB}

For all integers $a$ and $q\neq 0$, there uniquly exists a pair $(r,s)$ of integers such that 
\[
a=rq+s,\quad 0\leq s<|q|.
\] 
Then we define $a \mod q=s$.

\begin{lemma}[{see \cite[Lemma~2]{Dubickas2002} or \cite[Lemma~2.1]{Saito25ii}}]\label{Lemma:Periodic}Let $(b_k)_{k\ge 1}$ be a sequence of integers satisfying that there exist integers $a_{d-1},\ldots, a_0$ such that for every $k\in \mathbb{N}$ we have
\begin{equation}\label{eq:recurrence}
b_{k+d}=a_{d-1} b_{k+d-1}+a_{d-2}b_{k+d-2}+\cdots+a_0 b_k  
\end{equation}
Let $q$ be an integer. If $a_0$ and $q$ are coprime, then $(b_k \mod  q)_{k\ge 1}$ is purely periodic \textit{i.e.} there exists $L\in \mathbb{N}$ such that for all $k\in \mathbb{N}$ we have $b_k\equiv b_{k+L} \mod q$.
 \end{lemma}

\begin{proof}
If $q=1$, then it is trivial. Thus, we assume $q\neq 1$.  Since 
\[
( b_{k+d-1} \mod q, b_{k+d-2} \mod q, \ldots, b_k \mod q )\in \{0,1,\ldots, q-1\}^{d+1}
\]
for all $k=1,2,\ldots$, there are integers $r$ and $s$ with $r>s$ such that 
\begin{align*}
&( b_{r+d-1} \mod q, b_{r+d-2} \mod q, \ldots, b_r \mod q )\\
&=( b_{s+d-1} \mod q, b_{s+d-2} \mod q, \ldots, b_s \mod q ).
\end{align*}
Let $L=r-s$. Then, by \eqref{eq:recurrence} with $k=r$ and $k=s$, we obtain $b_{r+d} \equiv b_{s+d} \mod q$. Again applying \eqref{eq:recurrence} with $k=r+1$ and $k=s+1$, we have $b_{r+d+1} \equiv b_{s+d+1} \mod q$. By iterating this argument, we have $b_{k}\equiv b_{k+L} \mod q$ for all $k\geq s$. In addition, since $a_0$ and $q$ are coprime, there exists $u\in \{1,2,\ldots,q-1\}$ such that $a_0 u \equiv 1 \mod q$.  By \eqref{eq:recurrence}, we obtain 
\begin{equation}\label{eq:recurrence-modq}
b_k \equiv u b_{k+d}-ua_{d-1} b_{k+d-1}-ua_{d-2}b_{k+d-2}-\cdots-ua_1 b_{k+1}\quad \mod q
\end{equation}
for all $k\in \mathbb{N}$. Since $b_{k}\equiv b_{k+L} \mod q$ for all $s\leq k\leq s+d$,  \eqref{eq:recurrence-modq} with $k=s-1$ implies that $b_{s-1}=b_{s+L-1}$. By iterating this argument, we conclude that $b_k=b_{k+L}$ for all $k\in \mathbb{N}$. 
\end{proof}

\begin{proof}[Proof of Theorem~\ref{Theorem:Main2}]

Suppose that a sequence $(c_k)_{k\ge 1}$ satisfies \eqref{B1} to \eqref{B5} in Theorem~\ref{Theorem:Main2}. By choosing  $\theta=21/40$, $(c_k)_{k\ge 1}$ satisfies \eqref{G1} to \eqref{G4} in Proposition~\ref{Proposition:General} and the constant $\xi=\xi(C_k)$ exists.

Suppose that $\xi^{C_m}\in \mathbb{N}$ for some $m\in \mathbb{N}$. Then by \eqref{B5}, there exists $k\in I$ with $k>m$ such that $C_m\mid C_k$, and hence
\[
 \xi^{C_k}= (\xi^{C_m})^{C_k/C_m}
\] 
is a composite number, but $\xi^{C_k}=\lfloor \xi^{C_k}\rfloor$ is a prime number, a contradiction. Therefore, we have $\xi^{C_m}\notin \mathbb{N}$ for every $m\in \mathbb{N}$. 

Assume that $\xi$ is algebraic. Then,  we obtain \eqref{Result:General:1}, \eqref{Result:General:2}, \eqref{Result:General:3}, and \eqref{Result:General:4} in Proposition~\ref{Proposition:General}. Especially, by \eqref{Result:General:2}, there exists $g\in \mathbb{N}$ such that $\xi^g=\beta$ is a Pisot number of degree $\ell$ with 
\begin{equation}
2 \leq \ell\leq 1 + \biggl(\frac{19}{40}\limsup_{k\to \infty} c_{k+1}-1\biggr)^{-1}, 
\end{equation}
and for sufficiently large $k\in I$, we have $g\mid C_k$. 

Suppose that $\ell=2$. By \eqref{Result:General:3}, we have $\beta^g=(1+\sqrt{5})/2$ and $C_k/g$ is an odd prime number for sufficiently large $k$. Then, we take a sufficiently large $m\in I$ with $g\mid C_m$. By \eqref{B5}, there exists $k\in I$ with $k>m$ such that $C_m\mid C_k$. Therefore, $C_{m}/g$ is a divisor of $C_{k}/g$, a contradiction.

In addition, we now suppose \eqref{B6} and let us show that $\xi$ is transcendental. Let $\beta_1, \beta_2,  \ldots ,  \beta_\ell$ be all the conjugates over $\mathbb{Q}$ of $\beta$ labeled so that 
\[
\xi^g=\beta=\beta_1>1>|\beta_2|\geq  \cdots \geq |\beta_\ell|. 
\]
Take a sufficiently large integer $m\in I$ such that $g\mid C_m$. Let $\alpha_1=\beta_1^{C_{m}/g},\ldots,\alpha_\ell=\beta_\ell^{C_{m}/g}$. By Lemma~\ref{Lemma:PisotPowers}, $\alpha_1$ is a Pisot number of degree $\ell$. For every $n\in \mathbb{N}$, we define
\[
S(n)= \Tr(\alpha^n)=\alpha_1^{n} + \cdots + \alpha_\ell^{n}.
\] 
Since $\alpha_1$ is an algebraic integer of degree $\ell$, there exist integers $e_0, e_1,\ldots, e_\ell$ such that for all $j\in \{1,2,\ldots, \ell\}$, we have
\[
\alpha_j^\ell = e_{\ell-1}\alpha_j^{\ell-1}  +e_{\ell-2}\alpha_j^{\ell-2} \cdots+e_0,  
\]
where $e_0=\pm\alpha_1\cdots\alpha_\ell$.  Thus, for all non-negative integers $n$ and for all $j\in \{1,2,\ldots,\ell \}$
\[
\alpha_j^{n+\ell} = e_{\ell-1}\alpha_j^{n+\ell-1}  +e_{\ell-2}\alpha_j^{n+\ell-2} \cdots+e_0 \alpha_j^n.  
\]
By taking summation over $j\in \{1,2,\ldots, \ell\}$ on the both sides, we obtain 
\begin{equation}
S(n+\ell)=     e_{\ell-1}S(n+\ell-1) +e_{\ell-2}S(n+\ell-2)+ \cdots+e_0 S(n).
\end{equation}
Since $\beta$ is a Pisot number and $m$ is sufficiently large, we have 
\[
S(1)=\beta_1^{C_m/g} + \cdots +\beta_\ell^{C_m/g} > | \beta_1^{C_m/g}  \cdots  \beta_\ell^{C_m/g}|=|\alpha_1\cdots \alpha_\ell|=|e_0|.  
\]
We note that $S(1)$ is a prime number since $S(1)=\Tr(\beta^{C_m/g})=\lfloor \xi^{C_m}\rfloor$ by \eqref{Result:General:4}, and hence $S(1)$ and $e_0$ are coprime. Therefore, Lemma~\ref{Lemma:Periodic} with $(b_k)_{k\ge 1}\coloneqq (S(n))_{n\ge 1}$ and $q\coloneqq S(1)$ implies that $(S(n) \mod S(1))_{n\ge 1}$ has period $L$ for some $L\in \mathbb{N}$.  By \eqref{B6}, there exists $k\in I$ with $k>m$ such that $C_k\equiv C_m \mod LC_m$. Therefore, there exists $n\in \mathbb{N}$ such that
\[
C_k=LC_m n+C_m=(Ln+1)C_m,
\]  
and hence 
\[
\lfloor \xi^{C_k}\rfloor=\Tr(\beta^{C_k/g} )=\Tr(\alpha^{C_k/C_m} )  =S(C_k/C_m) =S(Ln+1)\equiv S(1) \equiv 0 \mod S(1). 
\]
This is a contradiction since $\lfloor \xi^{C_k}\rfloor>\lfloor \xi^{C_m}\rfloor=S(1)$ and $\lfloor \xi^{C_k}\rfloor$ is a prime number.
\end{proof}

\begin{proof}[Proof of Corollary~\ref{Corollary:Recurrence}]

Suppose that $(R_k)_{k\ge 1}$ is a sequence of positive integers satisfying \eqref{B1'} to \eqref{B4'}. Let $c_1=R_1$ and $c_{k+1}=R_{k+1}/R_k$ for all $k\in \mathbb{N}$. Then, it is clear that $(c_k)_{k\ge 1}$ satisfies \eqref{B1} to \eqref{B4} in Theorem~\ref{Theorem:Main2}. We note that \eqref{B5} follows from \eqref{B6}. Therefore, we shall show  \eqref{B6}. 

Let $L$ and $m$ be arbitrary positive integers.  By \eqref{B4'}, we have
\[
R_{k+d}=a_{d-1}R_{k+d-1}+a_{d-2} R_{k+d-2}+ \cdots + a_1 R_{k+1}+a_0R_{k} 
\]
for all $k\in \mathbb{N}$, where $a_0=\pm 1$. Therefore, since $a_0$ and $LR_m$ are coprime, Lemma~\ref{Lemma:Periodic} with $(b_k)_{k\ge 1}\coloneqq (R_k)_{k\ge 1}$ and $q\coloneqq LR_m$ yields that $(R_k \mod LR_m)_{k\ge 1}$ is purely periodic with period $U$ for some $U\in \mathbb{N}$, and hence choosing $k=m+U$, 
\[
R_{k}=R_{m+U} \equiv  R_m \mod LR_m.
\]
Thus, $\eqref{B6}$ is true. Therefore, Theorem~\ref{Theorem:Main2} implies that $\xi(R_k)$ exists and it is transcendental. 
\end{proof}

\section{Type C: Proof of Theorem~\ref{Theorem:Main3}}\label{Section:TypeC}

\begin{lemma}\label{Lemma:Const}
Let $c$ be a real number greater than or equal to $2$. Suppose that there exists $x_0\geq 1$ such that for all real numbers $x\geq x_0$ we have \eqref{P1}. Let $(c_k)_{k\ge 1}$ be a sequence of real numbers satisfing
\begin{itemize}
\item $c_1>0$\textup{;}
\item $c_{k+1}\geq c$ for all $k\in \mathbb{N}$.
\end{itemize}
Then, for all prime numbers $p> x_0^{1/c_2}$, there exists $A\in \mathcal{W}(C_k)$ such that $p=\lfloor A^{C_1}\rfloor$.
\end{lemma}

\begin{proof}

Let $p_1$ be an arbitrary prime number with $p_1>   x_0^{1/c_2}$. Then, by applying \eqref{P1} with $x=p_1^{c_2}\geq x_0$, there exists $p_2\in \mathcal{P}$ such that    
\begin{equation}\label{Ineq:Const:1}
p_1^{c_2}\leq p_2 <   p_1^{c_2}+cp_1^{c_2(1-1/c)}\leq p_1^{c_2}+c_2p_1^{c_2-1}.
\end{equation}
Let $f(x)=x^{c_2}$. Then by the Taylor expansion of $f(x)$ around $x=p_1$, there exists $\eta\in (p_1,p_1+1)$ such that $f(p_1+1)=f(p_1)+f'(p_1) + f''(\eta)/2$, and hence 
\begin{equation}\label{Ineq:Const:2}
(p_1+1)^{c_2}  = p_1^{c_2}+c_2 p_1^{c_2-1} + c_2(c_2-1) \eta^{c_2-2} > p_1^{c_2}+c_2 p_1^{c_2-1}+1.
\end{equation}
Combining \eqref{Ineq:Const:1} and \eqref{Ineq:Const:2}, we have $p_1^{c_2}\leq p_2 < (p_1+1)^{c_2}-1$. Let $m$ be an integer greater than or equal to $2$. By iterating this argument, we construct a sequence of prime numbers $(p_k)_{k\ge 1}$ such that
\begin{equation}\label{Ineq:Const:3}
p_{k}^{c_{k+1}}\leq p_{k+1} < (p_{k}+1)^{c_{k+1}}-1.
\end{equation}
Similarly with the proof of Theorem~\ref{Theorem:Matomaki}, the limit $\lim_{k\to \infty} p_k^{1/C_k}$ exists, say $A$. In addition, we have $A\in \mathcal{W}(C_k)$ and $\lfloor A^{C_1} \rfloor =p_1$. 
\end{proof}

\begin{proof}[Proof of Theorem~\ref{Theorem:Main3}]
Suppose that a sequence $(c_k)_{k\ge 1}$ satisfies \eqref{C1} to \eqref{C5}. Then,  $(c_k)_{k\ge 1}$ satisfies \eqref{B1} to \eqref{B5}. Thus, by Theorem~\ref{Theorem:Main2}, the number $\xi=\xi(C_k)$ exists, and either $\xi$ is transcendental or there exists $g$ such that $\xi^g=\beta$ is a Pisot number of degree $\ell$ with 
\begin{equation}\label{equation:typeC:0}
3\leq \ell \leq 1+ \left(\frac{19}{40}\limsup_{k\to \infty} c_{k+1} -1 \right)^{-1}\leq 1+ \left(\frac{57}{40} -1 \right)^{-1}=1+\frac{40}{17}<4
\end{equation}
and for every sufficiently large $k\in I$ we have $g\mid C_k$. Therefore, $\ell=3$. We note that  \eqref{C2} leads to
$I=\{k\in \mathbb{N}\colon c_{k+1}\geq 40/19+\epsilon \}=\mathbb{N}$, and hence $g\mid \mathrm{agcd}(C_k)$ since  $g\mid C_k $ for every sufficiently large $k\in \mathbb{N}$. 

By Bertrand's postulate (or Bertrand-Chebyshev theorem), we find a prime number $p_1\in [x_0^{1/c_2},2x_0^{1/c_2}]$. Therefore, by applying Lemma~\ref{Lemma:Const}, there exists $A\in \mathcal{W}(C_k)$ such that $\lfloor A^{C_1}\rfloor = p_1$. Since $\xi$ is the smallest real number of $\mathcal{W}(C_k)$, we obtain 
\[
\xi \leq A \leq (p_1+1)^{1/C_1} \leq (2x_0^{1/c_2}+1 )^{1/c_1}.
\]
Therefore, $\xi^g=\beta$ is a cubic Pisot number less than or equal to 
\[
(2x_0^{1/c_2}+1)^{g/c_1}.\qedhere
\]
\end{proof}

\section{Type D: Proof of Theorems~\ref{Theorem:Main5} and \ref{Theorem:Main4}}\label{Section:TypeD}

\begin{lemma}\label{Lemma:finiteT-expansion}
Let $(b_j)_{1\leq j\leq m}$ be a sequence of positive real numbers satisfying $b_j>1$ for all $j\in [1,m]$. Let $B_j=b_1\cdots b_j$ for every $j\in [1,m]$. Then, for every $x\in [0,1)$, there are a sequence $(a_j)_{1\leq j\leq m}$ of integers and a real number $x'\in[0,1)$ such that $0\leq a_j < b_j$ for all $j\in [1,m]$ and 
\begin{equation}\label{eq:C-expansion}
x=\sum_{j=1}^m  a_j B_j^{-1}+ x' B_m^{-1}.
\end{equation}
\end{lemma}

\begin{proof}
We define $x_1=x$, $a_1=\lfloor b_1 x_1\rfloor$, and $x_2=b_1x_1-a_1=b_1 x_1 -\lfloor b_1 x_1 \rfloor$. Then we see that $x= a_1B_1^{-1} + x_2 b_1^{-1}$. Recursively, we define $(a_j)_{1\leq j\leq m}$ and $(x_j)_{1\leq j\leq m+1}$ by
\[
a_j=\lfloor b_j x_j \rfloor\in[0,b_j),\quad  x_{j+1}=b_jx_j-a_j\in[0,1)
\]
for all $j\in [1,m]$. This implies that $x_j= a_jb_j^{-1} +x_{j+1}b_j^{-1}$, and hence 
\[
x=a_1B_1^{-1} + x_2 b_1^{-1}=a_1B_1^{-1} +a_2B_2^{-1} +x_3B_2^{-1}=\cdots=\sum_{j=1}^m a_j B_j^{-1} + x_{m+1} B_m^{-1}
\]
By setting $x'=x_{m+1}$, we obtain the lemma. 
\end{proof}

\begin{lemma}\label{Lemma:digit-restriction}
Let $\psi$ be a real number in $[0,1)$. Let $b_1$ be an arbitrary positive integer greater than $1$, and let $m$ be an integer greater than $2$. Let $(a_i)_{1\leq i\leq m}$ be as in Lemma~\ref{Lemma:finiteT-expansion} with $(b_j)_{1\leq j \leq m}\coloneqq (b_1,3,\ldots, 3)=b_1\ 3^{m-1}$ and $x\coloneqq \psi$. Suppose that
\begin{equation}\label{Relation:Range-psi}
\{B_j \psi\} \in [1/4, 3/4]
\end{equation}
for every $j\in[1, m-1]$. Then $(a_{j+1},a_{j+2})\in \{02,10,11,20\}$ for every $j\in [1,m-2]$.
\end{lemma}

\begin{proof} By Lemma~\ref{Lemma:finiteT-expansion}, for every $j\in [1,m-2]$, we have 
\[
(a_{j+1},a_{j+2})\in \{00,01,02,10,11, 20, 21, 22\}
\]
and
\begin{equation}\label{eq:digit-exp}
\{B_{j}\psi\}= \frac{a_{j+1}}{3}+\frac{a_{j+2}}{9} + \frac{\psi'}{9}.
\end{equation}
for some $\psi'=\psi'(j)\in [0,1)$.\\

\noindent\underline{Case $00$ or $01$}: If $(a_{j+1},a_{j+2})\in \{00,01\}$ for some $j\in [1,m-2]$, then \eqref{eq:digit-exp} implies that  
\[
\{B_{j}\psi\}\leq \frac{2}{9}<\frac{1}{4},
\]
a contradiction to \eqref{Relation:Range-psi}.

\noindent\underline{Case $21$ or $22$}: If $(a_{k+i},a_{k+i+1})\in \{21,22\}$ for some $i\in [1,m-2]$, then  \eqref{eq:digit-exp} implies that
\[
\{B_{j}\psi\}\geq \frac{2}{3} + \frac{1}{9} = \frac{7}{9} > \frac{3}{4},
\]
a contradiction to \eqref{Relation:Range-psi}. Therefore, we conclude the lemma.
\end{proof}

\begin{lemma}\label{Lemma:Liouville}
Let $\psi$, $m$, $(b_j)_{1\leq j\leq m}$, and $(a_j)_{1\leq j\leq m}$ be as in Lemma~\ref{Lemma:digit-restriction}. Let $\mu$ be the largest positive integer in $[1,m]$ such that $a_j=1$ for all $j\in[1,\mu]$, where $\mu=0$ if $a_1\neq 1$. Then, we have 
\begin{equation}\label{ineq:LiouvilleBound}
0\leq |2B_1 \psi - p |\leq  3^{-\mu+1}
\end{equation}
for some $p\in \mathbb{Z}$. Furthermore, if $\psi$ satisfies \eqref{ineq:non-Liouville} for some $\lambda>0$, then 
\[
\mu \leq \lambda \frac{\log B_1}{\log 3} +O(1),
\]
where the implicit constant depends only on $\lambda$.
\end{lemma}

\begin{proof} The lemma is clear for $\mu\in \{0,1\}$, and so we may assume that $\mu\ge 2$. By Lemma~\ref{Lemma:finiteT-expansion}, we obtain 
\[
\{B_1 \psi\}= \frac{1}{3} + \frac{1}{3^2} + \cdots +\frac{1}{3^{\mu-1}} + \frac{\psi'}{3^{\mu-1}}
\]
for some $\psi'\in [0,1)$. By the formula for geometric sums, we observe that
\[
\frac{1}{3} + \frac{1}{3^2} + \cdots +\frac{1}{3^{\mu-1}}=\frac{1}{3}\cdot \frac{1-(\frac{1}{3})^{\mu-1}}{1-\frac{1}{3}} = \frac{1}{2} - \frac{3}{2} \cdot \frac{1}{3^{\mu}},
\]
which implies \eqref{ineq:LiouvilleBound}. Furthermore, let us suppose that $\psi$ satisfies \eqref{ineq:non-Liouville} for some $\lambda>0$. Then, by applying \eqref{ineq:non-Liouville} with $N\coloneqq \max(2B_1,p)$, $q\coloneqq 2B_1$, and $p\coloneqq p$, we obtain  
\[
\max(2B_1, p)^{-\lambda}  \leq |2B_1 \psi - p |\leq 3^{-\mu+1}.
\]
By taking the logarithm on both sides and $p\ll B_1$, we have  
\[
\mu \leq \lambda \frac{\log \max(2B_1, p) }{\log 3}+1=\lambda \frac{\log B_1 }{\log 3}+O_\lambda (1). \qedhere
\]
\end{proof}

Hereinafter, we suppose that $(c_k)_{k\ge 1}$ is a sequence of positive integers satisfying \eqref{D1} and \eqref{D2} in Theorem~\ref{Theorem:Main4}. In addition, we suppose that $c_{k+1}=3$ for infinitely many $k$.  Let $\theta$ be a real number in $[1/2,1)$ satisfying \eqref{dag}. Then, $(c_k)_{k\ge 1}$ satisfies \eqref{G1} to \eqref{G4} in Proposition~\ref{Proposition:General}. Thus, $\xi=\xi(C_k)$ exists. We also obtain \eqref{ineq:key1-assump1} similarly to the proof of Theorem~\ref{Theorem:Main2}. We now assume that $\xi$ is algebraic, which will lead to a contradiction in the proof of Theorems~\ref{Theorem:Main5} and \ref{Theorem:Main4}.  

By Proposition~\ref{Proposition:General}, we have \eqref{Result:General:1} to \eqref{Result:General:4}. Therefore, by choosing $\theta=21/40$, we have 
\[
I=\{k\in \mathbb{N}\mid c_{k+1}\ge 1/(1-\theta)+\epsilon\}=\{k\in \mathbb{N}\mid c_{k+1}\ge 3\}.
\]
Then, there exists $g\in \mathbb{N}$ such that for every sufficiently large $k\in I$, we have the following properties:
\begin{enumerate}
\renewcommand{\theenumi}{\roman{enumi}'}
\renewcommand{\labelenumi}{(\theenumi)}
\item \label{i'}$\{\xi^{C_k}\}\ll \xi^{-C_k((1-\theta)c_{k+1}-1)}$;
\item \label{ii'}$\beta=\xi^g$ is a Pisot number of degree $\ell\in \{2,3\}$, and $g \mid C_k$;
\item \label{iii'}if $\ell=2$, then $C_k/g$ is an odd prime number;
\item \label{iv'}$\mathrm{Tr}(\beta^{C_k/g}) = \lfloor \beta^{C_k/g} \rfloor$.
\end{enumerate}
It is clear that $\ell=3$ from \eqref{iii'} since $C_k/g$ is not a prime number for sufficiently large $k\in I$. Therefore, $\beta$ is a cubic Pisot number.

Let $B_k=C_k/g$ for every $k\ge 1$. Then, $B_k$ is an integer for every sufficiently large $k\ge 1$ by \eqref{ii'}. Let $\beta_1=\beta$, $\beta_2$, $\beta_3$ be all the conjugates of $\beta$. We recall that $|\beta_2|<1$ and $|\beta_3|<1$ since $\beta$ is a Pisot number. Let $k$ be a sufficiently large variable running over $I$. Then,  \eqref{iv'} leads to
\begin{equation}\label{ineq:positive-fractional-part}
0\leq \{\beta_1^{B_k}\}= -\beta_2^{B_k}-\beta_3^{B_k}.
\end{equation}

\begin{proof}[Proof of Theorem~\ref{Theorem:Main4}] In this proof, we further suppose \eqref{D3} and \eqref{D4}.  

If $\beta_2$ is a real number, then $\beta_3$ is so. By \eqref{D3}, $B_k$ is even. Therefore, $\beta_2^{B_k}$ and $\beta_3^{B_k}$ are positive real numbers, a contradiction to \eqref{ineq:positive-fractional-part}. 

If $\beta_2$ is not a real number, then let $\beta_2=|\beta_2|e^{2\pi i \phi}$, where $\phi\in (0,1)$. This leads to $\beta_3=|\beta_2|e^{-2\pi i \phi}$. Therefore, \eqref{ineq:positive-fractional-part} implies that 
\begin{equation}\label{ineq:positive-fractional-part2}
0\leq -|\beta_2|^{B_k}(e^{2\pi i B_k \phi} +e^{-2\pi i B_k \phi})=-2 |\beta_2|^{B_k}\cos(2\pi B_k \phi),
\end{equation}
and hence
\begin{equation}\label{Relation:Range-phi}
\{B_k \phi\} \in [1/4,3/4]
\end{equation}
for every sufficiently large $k\in I$.  Let $k_1$ be a sufficiently large positive integer with $c_{k_1}=2$ and $c_{k_1+1}=3$. Then, by \eqref{D2} and \eqref{D3}, there exists a sequence $(\nu_{k})_{k\ge 1}$ of non-negative integers such that
\[
(c_k)_{k\ge k_1}= 2\ 3^{\nu_{1}}\ 2\ 3^{\nu_{2}}\ 2\ \cdots. 
\]
Take a sufficiently large positive integer $N\ge 1$ with $\nu_{N+1}\ge 3$. We note that such $N$ exists by \eqref{D4}. Let $t=k_1-1+\sum_{k=1}^N (\nu_{k}+1)$, and let $m=\nu_{N+1}$. Then, $c_{t+j+1}=3$ holds for every $j\in [1,m-1]$, and thus  
\[
\{B_{t+j}\phi\}\in [1/4,3/4] \quad (j\in [1,m-1]). 
\]
By Lemma~\ref{Lemma:digit-restriction} with $\psi\coloneqq \phi$, $b_1\coloneqq B_{t+1}$, $m\coloneqq m$, we obtain 
\begin{equation}\label{Relation:a}
(a_{t+j+1},a_{t+j+2})\in\{02, 10, 11,20\}
\end{equation}
for every $j\in [1, m-2]$, where $(a_{t+j+1})_{1\leq j\leq m-1}$ denotes as in Lemma~\ref{Lemma:finiteT-expansion} with $(b_j)_{1\leq j\leq m} \coloneqq (B_{t+1}, 3, \ldots, 3)$ and $x\coloneqq \phi$. 
Let $\lambda$ be as in Lemma~\ref{Lemma:non-Liouville} with $\psi\coloneqq \phi$.

Let $\mu$ be the largest integer such that $a_{t+j+1}=1$ for every $j\in [1, \mu]$, where we define $\mu=0$ if $a_{t+j+1}\neq 1$. Then, Lemma~\ref{Lemma:Liouville} with $B_1\coloneqq B_{t+1}$ implies that
\[
\mu \leq \lambda \frac{\log B_{t+1}}{\log 3} +O(1).
\]
By the definition of $\mu$ and \eqref{Relation:a}, we obtain $(a_{t+j+1})_{\mu+1\leq i\leq m-1} =  02020\cdots$,  and hence
\[
\{B_{t+\mu+1} \phi\}= \frac{1}{4}+O(3^{-(m-\mu)} ).
\]
We find integers $p$ and $q$ such that $\max(|p|,|q|)\ll B_{t+\mu+1}$ and
$|q\phi -p|\ll 3^{-(m-\mu)}$. Furthermore, by  Lemma~\ref{Lemma:non-Liouville}, we have
\[
B_{t+\mu+1}^{-\lambda} \ll |q\phi -p|\ll 3^{-(m-\mu)}.
\] 
By taking logarithms, we have
\[
m-\mu \leq \frac{ \lambda \log B_{t+\mu+1} }{\log 3}, 
\]
and the right-hand side is less than or equal to 
\[
\lambda \left(\mu + \sum_{k=1}^N \nu_k + \frac{\log 2}{\log 3} N +O(1)\right)
\]
since  $t=k_1-1+\sum_{k=1}^N (\nu_{k}+1)$ and $ B_{t+\mu+1}=B_{k_1-1}\cdot 2\cdot 3^{\nu_1} \cdots  2\cdot 3^{\nu_N}2 \cdot 3^{\mu} $. Recalling $m=\nu_{N+1}$, we obtain
\[
\nu_{N+1} \leq (\lambda^2+2\lambda)  \left(\sum_{k=1}^N \nu_k + \frac{\log 2}{\log 3} N \right)+O(1),
\]
a contradiction to \eqref{D4} by taking $\limsup_{N\to \infty}$. Therefore, $\xi$ is transcendental. 
\end{proof}

\begin{proof}[Proof of Theorem~\ref{Theorem:Main5}] 
Instead of \eqref{D3} and \eqref{D4}, we suppose that there exists $k_0>0$ such that for every $k\ge k_0$, we have $c_{k+1}=3$. Then, $I=\{k\in \mathbb{N} \colon k\ge k_0 \}$ 

If $\beta_2$ is not a real number, then let $\beta_2=|\beta_2|e^{2\pi i \phi}$, where $\phi\in (0,1)$. Similarly with the proof of Theorem~\ref{Theorem:Main4}, there exists $k_1>0$ such that for every $k\ge k_1\ge k_0$, we obtain
\begin{equation}\label{Relation:Range-phi2}
\{B_k \phi\} \in [1/4,3/4].
\end{equation}
By Lemma~\ref{Lemma:digit-restriction} with $\psi\coloneqq \phi$, $b_1\coloneqq B_{k_1}$,  we obtain 
\[
(a_{k_1+j},a_{k_1+j+1})\in\{02, 10, 11,20\} 
\]
for every positive integer $j$. Therefore, $\phi$ must be a rational number since $(a_{k})_{k\ge 1}$ is ultimately periodic and $B_{k+1}/B_k=c_{k+1}=3$ for every $k\ge k_1$. This contradicts Lemma~\ref{Lemma:Mignotte}, and hence $\beta_2$ and $\beta_3$ are real numbers.

We note that Lemma~\ref{Lemma:Smyth} implies that $|\beta_2|>|\beta_3|$. Let $\theta$ be a real number in $(1/2,2/3)$ satisfying \eqref{dag}. Then by combining \eqref{ineq:positive-fractional-part} and \eqref{i'}, we have
\[
 |\beta_2|^{B_k} \left(1- \left(\frac{|\beta_3|}{|\beta_2|}\right)^{B_k}\right)\leq  \|\beta_1^{B_k} \|\ll \beta_1^{(3\theta-2)B_k}. 
\]
This implies that $|\beta_2|\leq  \beta_1^{3\theta-2}$ by taking logarithms on both sides and $k\to \infty$. Therefore, by setting $\eta= \frac{\log |\beta_3|}{\log |\beta_2|}>1$, we have
\[
1\leq |\beta_1\beta_2\beta_3|= |\beta_1| |\beta_2|^{1+\eta}  \leq |\beta_1|^{1+ (3\theta-2)(1+\eta)},
\]
which yields that $1+ (3\theta-2)(1+\eta)\ge 0$. It is equivalent to $\eta\leq (3\theta-1)/(2-3\theta)$, and hence 
\begin{equation}\label{Inequality:beta2beta3}
|\beta_3|<-\beta_2 \leq \min (|\beta_3|^{(2-3\theta)/(3\theta-1)}, \beta_1^{3\theta-2})
\end{equation}
where $|\beta_2|=-\beta_2$ by \eqref{ineq:positive-fractional-part}. Therefore, we conclude \eqref{Inequality:Theorem:Main5} since  \eqref{dag} is true for $\theta=21/40$ by Theorem~\ref{Theorem:BHP}.

Furthermore, let us assume that  \eqref{dag} is true for every fixed $\theta>1/2$.  We now choose a real number $\theta$ satisfying
\begin{equation}\label{Inequality:Choosetheta}
\frac{1}{2} < \theta < \frac{1}{2} + \frac{1}{3}\left(\frac{1}{2} - \frac{1}{1+\eta} \right)= \frac{1}{3}\left(2-\frac{1}{1+\eta}\right).
\end{equation}
Then, by \eqref{Inequality:Choosetheta}, we observe that 
\[
\frac{2-3\theta}{3\theta-1}=\frac{1}{3\theta-1}-1> \frac{1}{1- \frac{1}{1+\eta}}-1=\frac{1}{\eta}=\frac{\log|\beta_2|}{\log|\beta_3|}.
\]
Combining this and \eqref{Inequality:beta2beta3} implies that 
\[
|\beta_2 |\leq |\beta_3|^{(2-3\theta)/(3\theta-1)}< |\beta_3|^ {(\log|\beta_2|)/(\log|\beta_3|)}=|\beta_2|, 
\]
a contradiction. Therefore, $\xi$ is transcendental. 
\end{proof}

\appendix

\section{Proof of Corollary \ref{Corollary:AssumeRH}}\label{AppendixA}

\begin{lemma}\label{Lemma:Akiyama}
There are exactly $20$ cubic Pisot numbers less than or equal to $3$. Let $\alpha_1,\ldots, \alpha_{20}$ be such Pisot numbers labeled so that $1<\alpha_1<\cdots< \alpha_{20} <3$. Then, we obtain the list of $\alpha_j$ $(j=1,2,\ldots, 20)$ as Table~\ref{Table:Pisot}. 
\end{lemma}
\begin{proof}
Let $M$ be a given positive integer. If $\alpha$ is a cubic Pisot number with $1<\alpha \leq M$, then the minimal polynomial $f(X)$ of $\alpha$ over $\mathbb{Q}$ forms as 
\[
f(X) =X^3-a_2 X^2 -a_1 X-a_0,\quad (a_2,a_1,a_0\in \mathbb{Z},\ a_0 \neq 0 ).
\]
Let $\alpha'$ and $\alpha''$ be the conjugates of $\alpha$. Then, we have $|\alpha'|<1$ and $|\alpha''|<1$ since $\alpha$ is a Pisot number. Therefore, the relation between roots and coefficients implies that 
\begin{gather*}
a_2=\alpha+\alpha'+ \alpha''\in (-1 ,M+2), \\
|a_1|=|\alpha\alpha'+ \alpha'\alpha''+\alpha\alpha''| <2M+1,\\
0<|a_0|=|\alpha \alpha'\alpha''| <M.    
\end{gather*}
By substituting $M=3$, we have $a_2 \in [0,4]$, $|a_1|\leq 6$, and $1\leq |a_0|\leq 2$. Thus, there are $5\cdot 13\cdot 4=260$ candidates of polynomials whose roots contain a cubic Pisot number. By Akiyama's result  \cite[Lemma~1]{Akiyama}, $f(X)$ is the minimal polynomial of some cubic Pisot number if and only if  
\begin{equation}\label{equation-Akiyama}
|a_1-1|<a_{2}+a_0\quad \text{and}\quad  a_0^2- a_1 < \mathrm{sgn}(a_0)(1+ a_2a_0),
\end{equation}
where $\mathrm{sgn}(x)=1$ if $x>0$; $\mathrm{sgn}(x)=-1$ if $x<0$. Using Mathematica, for all  $a_2 \in [0,4]$, $|a_1|\leq 6$, and $1\leq |a_0|\leq 2$, we check whether $(a_0,a_1,a_2)$ satisfies \eqref{equation-Akiyama}. After that, we examine whether all real roots of $f(X)$ are less than or equal to $3$, and obtain Table~\ref{Table:Pisot}. 
\end{proof}

\begin{center}
\begin{table}[htbp]\renewcommand{\arraystretch}{1.2}
\begin{tabular}{c|c|c}
$j$ & $\alpha_j$ & Minimal polynomial of $\alpha_j$  \\ \hline\hline
1 & $1.32471795724474602596\cdots$ & $X^3  -X -1$ \\ \hline
2 & $1.46557123187676802665\cdots$ & $X^3 -X^2  -1$ \\ \hline
3 & $1.75487766624669276004\cdots$ & $X^3 -2X^2 +X -1$ \\ \hline
4 & $1.83928675521416113255\cdots$ & $X^3 -X^2 -X -1$ \\ \hline
5 & $2.14789903570478735402\cdots$ & $X^3 -X^2 -2X -1$ \\ \hline
6 & $2.20556943040059031170\cdots$ & $X^3 -2X^2  -1$ \\ \hline
7 & $2.24697960371746706105\cdots$ & $X^3 -2X^2 -X +1$ \\ \hline
8 & $2.26953084208114277085\cdots$ & $X^3 -X^2 -2X -2$ \\ \hline
9 & $2.32471795724474602596\cdots$ & $X^3 -3X^2+ 2X -1$ \\ \hline
10 & $2.35930408597177642073\cdots$ & $X^3 -2X^2  -2$ \\ \hline
11 & $2.51154714169453198401\cdots$ & $X^3 -X^2 -3X -2$ \\ \hline
12 & $2.52137970680456756960\cdots$ & $X^3 -3X^2 +2X -2$ \\ \hline
13 & $2.54681827688408207913\cdots$ & $X^3 -2X^2 -X -1$ \\ \hline
14 & $2.65896708191699407934\cdots$ & $X^3 -2X^2 -X -2$ \\ \hline
15 & $2.76929235423863141524\cdots$ & $X^3 -3X^2 +X -1$ \\ \hline
16 & $2.83117720720833690413\cdots$ & $X^3 -2X^2 -2X -1$ \\ \hline
17 & $2.83928675521416113255\cdots$ & $X^3 -4X^2 +4X -2$ \\ \hline
18 & $2.87938524157181676810\cdots$ & $X^3 -3X^2  +1$ \\ \hline
19 & $2.89328919630449778890\cdots$ & $X^3 -3X^2 +X -2$ \\ \hline
20 & $2.91963956583941814511\cdots$ & $X^3 -2X^2 -2X -2$ \\ \hline
\end{tabular}
\vspace{5pt}
\caption{All cubic Pisot numbers $\leq 3$}\label{Table:Pisot}
\end{table}
\end{center}

\begin{proof}[Proof of Corollary~\ref{Corollary:AssumeRH}] 
Assume that \eqref{P1} is true for $x_0=1$. Similarly with the proof of \eqref{TypeD} in Theorem~\ref{Theorem:main:example}, either $\xi=\xi(r3^k-1)$ is transcendental or there exists $g\in \{1,2\}$ such that $\xi^g$ is a cubic Pisot number satisfying 
\[
\beta=\xi^g\leq (2x_0^{1/c_2} +1 )^{g/c_1}=3^{g/(3r-1)},
\]
where $g\in \{1,2\}$. Assume that $\xi$ is algebraic. Then, in the case $r\geq 3$, we have 
\[
\beta\leq 3^{2/(3r-1)}\leq 3^{1/4}<1.320<1.3247< \kappa, 
\]
a contradiction since $\kappa$ is the smallest Pisot number. Therefore, $\xi=\xi(r3^k-1)$ is transcendental for all $r\geq 3$.  

In the case $r=2$, we have $g=1$ by \eqref{equation:typeC:1}. Therefore, we obtain
\[
\beta\leq  3^{1/(3\cdot 2-1)}=3^{1/5}=1.2457\cdots < 1.3247 < \kappa,
\]
a contradiction.

In the case $r=1$ and $g=1$, we see that 
\[
\beta=\xi^g(=\xi)\leq  3^{1/(3-1)}=3^{1/2}=1.732\cdots < 1.754 < \alpha_3.
\]
This yileds that $\xi\in \{\alpha_1,\alpha_2\}$. By the definition of $\xi$, both $\lfloor \xi^{3^1-1}\rfloor$ and $\lfloor \xi^{3^2-1}\rfloor$ are prime numbers, but we observe that 
\[
\lfloor \alpha_1^{3^1-1} \rfloor = 1,
\quad
\lfloor \alpha_2^{3^2-1} \rfloor = 21,
\]
a contradiction. 

In the case $r=1$ and $g=2$, we see that 
\[
\beta=\xi^g(=\xi^2)\leq  3^{2/(3-1)}=3,
\]
and hence $\xi\in \{\alpha_1^{1/2},\alpha_2^{1/2},\ldots, \alpha_{20}^{1/2} \}$. Since $\lfloor (\alpha_j^{1/2})^{3^1-1} \rfloor=1$ for all $j\in \{1,2,3,4\}$, we have $\xi\in \{\alpha_5^{1/2},\alpha_6^{1/2},\ldots , \alpha_{20}^{1/2} \}$. Let us recall that $\lfloor \xi^{3^k-1} \rfloor$  is a prime number for every $k\in \mathbb{N}$, but we observe that
\begin{align*}
\lfloor(\alpha_5^{1/2})^{3^2 -1}\rfloor&=21=3\cdot 7,&
\lfloor(\alpha_6^{1/2})^{3^3 -1}\rfloor&=29226=2\cdot 3\cdot 4871,\\
\lfloor(\alpha_7^{1/2})^{3^2 -1}\rfloor&=25=5^2,&
\lfloor(\alpha_8^{1/2})^{3^2 -1}\rfloor&=26=2\cdot 13,\\
\lfloor(\alpha_9^{1/2})^{3^4 -1}\rfloor&=2\cdot 3\cdot 43\cdot 1750616861141, &
\lfloor(\alpha_{10}^{1/2})^{3^2 -1}\rfloor&=30=2\cdot 3\cdot 5,\\
\lfloor(\alpha_{11}^{1/2})^{3^2 -1}\rfloor&=39=3\cdot 13,&
\lfloor(\alpha_{12}^{1/2})^{3^2 -1}\rfloor&=40=2^3\cdot 5, \\
\lfloor(\alpha_{13}^{1/2})^{3^2 -1}\rfloor&=42=2\cdot 3\cdot 7,& 
\lfloor(\alpha_{14}^{1/2})^{3^2 -1}\rfloor&=49=7^2,\\
\lfloor(\alpha_{15}^{1/2})^{3^2 -1}\rfloor&=58=2\cdot 29,&
\lfloor(\alpha_{16}^{1/2})^{3^2 -1}\rfloor&=64=2^6,\\
\lfloor(\alpha_{17}^{1/2})^{3^2 -1}\rfloor&=64=2^6,&
\lfloor(\alpha_{18}^{1/2})^{3^2 -1}\rfloor&=68=2^2\cdot 17,\\
\lfloor(\alpha_{19}^{1/2})^{3^2 -1}\rfloor&=70=2\cdot 5\cdot 7,&
\lfloor(\alpha_{20}^{1/2})^{3^2 -1}\rfloor&=72=2^3\cdot 3^2.
\end{align*}
All of them are composite numbers, a contradiction. 
\end{proof}

We give a program using Mathematica to obtain a list of all cubic Pisot numbers up to a given positive real number $M$. We apply it to describe Table~\ref{Table:Pisot}. 

\begin{verbatim}
(*Input M as a positive real number*) 
M = Input["Enter a positive real number"];

(*Fix a range of coefficients (a0,a1,a2)*)
ceilM = Ceiling[M];
coeffs = 
  Flatten[Table[{a0, a1, a2}, {a0, -ceilM + 1, 
     ceilM - 1}, {a1, -2*ceilM, 2*ceilM}, {a2, 0, ceilM + 1}], 2];
coeffs = DeleteCases[coeffs, {0, _, _}];

(*Select polynomials f(x)=x^3-a2*x^2-a1*x-a0 of which all real roots are 
less than or equal to M, and make a list of {alpha, a0, a1, a2}, 
where alpha is the maxmum real root of f(x) (it should be a Pisot number)*)
list = {};
max = Length[pisotcoeffs];
For[j = 1, j <= max, j++,
  a0 = pisotcoeffs[[j, 1]];  
  a1 = pisotcoeffs[[j, 2]]; 
  a2 = pisotcoeffs[[j, 3]];  
  poly = x^3 - a2*x^2 - a1*x - a0;
  realroots = x /. NSolve[poly == 0, x, Reals, WorkingPrecision -> 50];
	If[Select[realroots, # > M &] == {},   
	list = Join[list, {{Max[realroots], a0, a1, a2}}] ];
];

(*Sort the list*)
list = SortBy[list, First];
max = Length[list];

(*Print*)
Print[StringForm[
   "The number of cubic Pisot numbers less than or equal to `` is 
exactly ``, and the list is as follows:\n
a cubic Pisot number, its minimal polynomial", M, max]];

For[j = 1, j <= max, j++,
 alpha = list[[j, 1]]; 
 a0 = list[[j, 2]]; a1 = list[[j, 3]]; a2 = list[[j, 4]];
 poly = x^3 - a2*x^2 - a1*x - a0;
 Print[StringForm[" alpha_`` = `` , ``", j, NumberForm[alpha, 25], 
   TraditionalForm[poly]]]
 ]
\end{verbatim}

\section*{Acknowledgement}
The author is grateful to Prof.~Art\={u}ras Dubickas and Prof.~Pieter Allaart for their helpful comments and advice. The author would like to thank Prof. Hajime Kaneko for the invitation to speak at Numeration and Substitution 2025 at the University of Tsukuba. The author had valuable conversations concerning Mills' constant with the participants at the conference. The author would like to thank Prof. Wataru Takeda for his encouraging comments and for finding some typos. The author is also grateful to Daniel Johnston for pointing out the best-known result on \eqref{P1}. The author was supported by JSPS KAKENHI Grant Number JP25K17223.


\begin{thebibliography}{MTY24}

\bibitem[AD04]{AlkauskasDubickas}
G.~Alkauskas and A.~Dubickas, \emph{Prime and composite numbers as integer
  parts of powers}, Acta Math. Hungar. \textbf{105} (2004), no.~3, 249--256.

\bibitem[Aki00]{Akiyama}
S.~Akiyama, \emph{Cubic {Pisot} units with finite beta expansions},
  Algebraic number theory and diophantine analysis. Proceedings of the
  international conference, Graz, Austria, August 30--September 5, 1998,
  Berlin: Walter de Gruyter, 2000, pp.~11--26.

\bibitem[Bak75]{Baker}
A.~Baker, \emph{Transcendental number theory}, Cambridge University Press,
  London-New York, 1975. 

\bibitem[BD08]{BugeaudDubickas}
Y.~Bugeaud and A.~Dubickas, \emph{On a problem of {Mahler} and
  {Szekeres} on approximation by roots of integers}, Mich. Math. J. \textbf{56}
  (2008), no.~3, 703--715.

\bibitem[BHP01]{BakerHarmanPintz}
R.~C.~Baker, G.~Harman, and J.~Pintz, \emph{The difference between consecutive
  primes. {II}}, Proc. London Math. Soc. (3) \textbf{83} (2001), no.~3,
  532--562. 

\bibitem[CC05]{CaldwellCheng}
C.~K.~Caldwell and Y.~Cheng, \emph{Determining {M}ills' constant and a
  note on {H}onaker's problem}, J. Integer Seq. \textbf{8} (2005), no.~4,
  Article 05.4.1, 9. 

\bibitem[CH23]{Cully-Hugill}
M.~Cully-Hugill, \emph{Primes between consecutive powers}, J. Number Theory \textbf{247} (2023), 100--117.

\bibitem[CJ23]{CullyHugillJohnston1}
M.~Cully-Hugill, D.~R.~Johnston,  \emph{On the error term in the explicit formula of Riemann-von Mangoldt}, Int. J. Number Theory \textbf{19} (2023), no.~6, 1205--1228.

\bibitem[CJ25]{CullyHugillJohnston2}
M.~Cully-Hugill, D.~R.~Johnston, \emph{On the error term in the explicit formulaof Riemann-von Mangoldt II}, Funct. Approx. Comment. Math.  Advance Publication (2025), 1--20. 


\bibitem[CMS19]{CMS}
E.~Carneiro, M.~B.~Milinovich, and K.~Soundararajan, \emph{Fourier
  optimization and prime gaps}, Comment. Math. Helv. \textbf{94} (2019), no.~3,
  533--568.

\bibitem[CZ04]{CorvajaZannier}
P.~Corvaja and U.~Zannier, \emph{On the rational approximations to the
  powers of an algebraic number: solution of two problems of {M}ahler and
  {M}end\`es {F}rance}, Acta Math. \textbf{193} (2004), no.~2, 175--191.
  
\bibitem[Dub02]{Dubickas2002}
A.~Dubickas, \emph{Integer parts of powers of {P}isot and {S}alem
  numbers}, Arch. Math. (Basel) \textbf{79} (2002), no.~4, 252--257.

\bibitem[Dub04]{Dubickas2004}
A.~Dubickas, \emph{Conjugate algebraic numbers close to a symmetric set}, St. Petersbg. Math. J. \textbf{16} (2005), No. 6, 1013--1016; translation from Algebra Anal. \textbf{16} (2004), No. 6, 123--127.

  
\bibitem[Dub22]{Dubickas2022}
A.~Dubickas, \emph{Transcendency of some constants related to integer
  sequences of polynomial iterations}, Ramanujan J. \textbf{57} (2022), no.~2,
  569--581. 

\bibitem[Dud69]{Dudley}
U.~Dudley, \emph{History of a formula for primes}, Amer. Math. Monthly
  \textbf{76} (1969), 23--28. 

\bibitem[Fin03]{Finch}
S.~R.~Finch, \emph{Mathematical constants}, Encyclopedia of Mathematics and
  its Applications, vol.~94, Cambridge University Press, Cambridge, 2003.
  
\bibitem[Ing37]{Ingham} 
A.~E.~Ingham, On the difference between consecutive primes, 
Quart. J. Math. (Oxford Ser.) \textbf{8} (1937), 255--266.
  
  
\bibitem[Mat07]{Matomaki2007}
K.~Matom\"{a}ki, \emph{Large differences between consecutive primes}, Q. J.
  Math. \textbf{58} (2007), no.~4, 489--518. 

\bibitem[Mat10]{Matomaki}
K.~Matom\"{a}ki, \emph{Prime-representing functions}, Acta Math. Hungar.
  \textbf{128} (2010), no.~4, 307--314. 

\bibitem[Mig84]{Mignotte}
M.~Mignotte, \emph{Sur les conjugu\'es des nombres de {P}isot}, C. R.
  Acad. Sci. Paris S\'er. I Math. \textbf{298} (1984), no.~2, 21. 

\bibitem[Mil47]{Mills}
W.~H.~Mills, \emph{A prime-representing function}, Bull. Amer. Math. Soc.
  \textbf{53} (1947), 604. 

\bibitem[MTY24]{MTY}
M.~J.~Mossinghoff, T.~S. Trudgian, and A. Yang, \emph{Explicit
  zero-free regions for the {Riemann} zeta-function}, Res. Number Theory
  \textbf{10} (2024), no.~1, 27, Id/No 11.

\bibitem[Sai25a]{Saito25ii}
K.~Saito, \emph{Intervals without primes near an iterated linear recurrence
  sequence}, preprint (2025), available at \url{https://arxiv.org/abs/2504.14968}.

\bibitem[Sai25b]{Saito25}
K.~Saito, \emph{Mills' constant is irrational}, Mathematika \textbf{71}
  (2025), no.~3, Paper No. e70027.

\bibitem[Sie44]{Siegel}
C.~L.~Siegel, \emph{Algebraic integers whose conjugates lie in the unit
  circle}, Duke Math. J. \textbf{11} (1944), 597--602. 

\bibitem[Smy75]{Smyth}
C.~J.~Smyth, \emph{Problems and {S}olutions: {S}olutions of {A}dvanced
  {P}roblems: 5931}, Amer. Math. Monthly \textbf{82} (1975), no.~1, 86.
 

\bibitem[ST22]{SaitoTakeda}
K.~Saito and W.~Takeda, \emph{Topological properties and algebraic
  independence of sets of prime-representing constants}, Mathematika
  \textbf{68} (2022), no.~2, 429--453.

\bibitem[WZ21]{WagnerZiegler}
S.~Wagner and V.~Ziegler.
\emph{Irrationality of growth constants associated with polynomial recursions}, 
J. Integer Seq. \textbf{24} (2021), no.~1, Art. 21.1.6, 9.

\bibitem[Wri51]{Wright51}
E.~M.~Wright, \emph{A prime-representing function}, Amer. Math. Monthly
  \textbf{58} (1951), 616--618. 

\end{thebibliography}
\end{document}